\documentclass[smallextended]{svjour3}

\usepackage[utf8]{inputenc}
\usepackage{amsmath,amssymb}
\usepackage{graphicx}
\usepackage[english]{babel}
\usepackage{mathptmx}
\usepackage{listings}
\usepackage{tikz}
\tikzset{>=latex}

\usepackage{pgfplots}

\newcommand{\R}{{\mathbb R}}

\newcommand{\ud}{{\mathrm d}}
\newcommand{\norm}[1]{\left\lVert#1\right\rVert}
\newcommand{\abs}[1]{\left\lvert#1\right\rvert}
\newcommand{\inprod}[1]{\left(#1\right)}
\newcommand{\vect}[1]{\mathbf{#1}}
\newcommand{\eye}{{\mathrm I}}
\newcommand{\widebar}[1]{\overline{#1}}
\newcommand{\Lap}{\Delta}

\def\qedhere{\quad\qed}

\DeclareMathOperator*{\argmin}{arg~min}

\newcounter{assump}
\newenvironment{assumptions}
                {\begin{list}{{\upshape A\arabic{assump}.}}{\usecounter{assump}
                \setlength{\leftmargin}{0pt}
                \setlength{\itemindent}{38pt}}}{\end{list}}

\newcommand{\TheTitle}{Backward step control for Hilbert space problems}
\newcommand{\TheAuthors}{A. Potschka}

\titlerunning{\TheTitle}
\authorrunning{\TheAuthors}

\title{\TheTitle}

\author{Andreas Potschka}
\institute{Interdisciplinary Center for Scientific Computing, Heidelberg
University, Im Neuenheimer Feld 205, 69120 Heidelberg, Germany.
\email{potschka@iwr.uni-heidelberg.de}}

\begin{document}
\maketitle

\begin{abstract}
We analyze backward step control globalization for finding zeros of G\^{a}\-teaux-differentiable functions that map from a Banach space to a Hilbert space. The results include global convergence to a distinctive solution characterized by propagating the initial guess by a generalized Newton flow with guaranteed bounds on the discrete nonlinear residual norm decrease and an (also numerically) easily controllable asymptotic linear residual convergence rate. The convergence theory can be exploited to construct efficient numerical methods, which we demonstrate for the case of a Krylov--Newton method and an approximation-by-discretization multilevel framework. Both approaches optimize the asymptotic linear residual convergence rate, either over the Krylov subspace or through adaptive discretization, which in turn yields practical and efficient stopping criteria and refinement strategies that balance the nonlinear residuals with the relative residuals of the linear systems. We apply these methods to the class of nonlinear elliptic boundary value problems and present numerical results for the Carrier equation and the minimum surface equation.
\end{abstract}

\keywords{Newton-type methods, globalization, Hilbert space, backward step control}

\subclass{65J15, 58C15, 65F08, 35J66, 74S05}

\section{Introduction}
\label{sec:introduction}

Let $U$ be a Banach space with norm $\norm{.}_U$ and $V$ be a
Hilbert space (we discuss generalizations to Banach spaces in
section~\ref{sec:Banach}) with inner product $\inprod{.,.}_V$ and norm
$\smash[t]{\norm{v}_V = \smash[b]{\sqrt{\inprod{v,v}_V}}}$ for $v \in V$. 
For some open subset $D \subseteq U$, let $F: D \to V$ be continuously
G\^ateaux differentiable~\cite[3.1.1.~Def.]{Hamilton1982} with derivative $F': D
\times U \to V$, i.e., 
\[
  F'(u, \delta u) = \lim_{h \to 0} \frac{1}{h} \left( F(u+h \delta u) - F(u)
  \right)
\]
exists for all $u \in D$, $\delta u \in U$ and $F'$ is continuous as a function
from the product space $D \times U$ to $V$ (which is a weaker requirement than
continuity of $F'$ considered as a map from
$D$ to $\mathcal{L}(U, V)$, the Banach space of all bounded linear operators
from $U$ to $V$). Note that continuous G\^ateaux differentiability implies
linearity of $F'$ in the second argument \cite[3.2.5~Thm.]{Hamilton1982}.
In the following, we write $A(u) \delta u$ shorthand for $A(u, \delta u)$
whenever an operator $A$ is linear in the second argument.
We consider the problem of finding an unknown $u \in D$ such that
\begin{equation}
  \label{eqn:FOfxIsZero}
  F(u) = 0_V
\end{equation}
with a Newton-type iteration: Given $u_0 \in D$, find a suitable approximation
$M: D \times V \to U$ (linear in the second argument) of the inverse of $F'(u)$
and a step size sequence $(t_k)_{k \in \mathbb{N}}$ satisfying $t_k \in [0, 1]$
such that the iteration
\begin{equation}
  \label{eqn:Newton}
  u_{k+1} = u_k + t_k \delta u_k \quad 
  \text{with } \delta u_k = -M(u_k) F(u_k)
\end{equation}
converges to a solution $u^\ast \in D$ of \eqref{eqn:FOfxIsZero}. 
The first and main part of this article is devoted to finding a suitable step
size sequence $(t_k)_{k \in \mathbb{N}}$ in section~\ref{sec:convergence}. Out of the
many ways to construct $M(u)$, we elaborate on two choices in
section~\ref{sec:designM}.
For convenience, we define the negative generalized Newton flow $f: D \to U$ as
\[
f(u) = M(u)F(u).
\]
As in \cite{Potschka2016}, our convergence analysis will be based on generalized
Newton paths $u^k: [0, \infty) \to U$, which are defined as the solutions of the
initial value problems
\begin{equation}
  \label{eqn:genDavidenko}
  \frac{\ud u^k}{\ud t}(t) = -f(u^k(t)) \quad 
  \text{for } t \in [0, \infty) \quad 
  \text{with } u^k(0) = u_k.
\end{equation}
We shall prove existence and uniqueness of solutions to \eqref{eqn:genDavidenko}
in our setting in Theorem \ref{thm:arclength}.
Note that \eqref{eqn:Newton} is an explicit Euler discretization of
\eqref{eqn:genDavidenko} with step sizes $(t_k)_{k \in \mathbb{N}}$. 
If $M$ is chosen as the inverse of $F'$ then \eqref{eqn:Newton} is a damped
Newton method and \eqref{eqn:genDavidenko} is the Davidenko differential
equation \cite{Davidenko1953}. The equivalence of the Newton method with
explicit Euler on the Davidenko differential equation has been exploited by
various authors (see, e.g.,
\cite{Deuflhard1974,Brezinski1975,Ascher1987,Bock2000b,Deuflhard2006}).

In the theory and the numerics below, the operator $M$ does not appear
explicitly anymore and only the function $f$ will be required, which implicitly
defines $M(u)$ in the direction $F(u)$. As it turns out, all other directions of
$M(u)$ are not important. This observation leads to crucial improvements in the
assumptions stated in \cite{Potschka2016}.

The convergence theory below lends itself immediately to the construction of
numerical algorithms for the approximate solution of \eqref{eqn:FOfxIsZero} via
\eqref{eqn:Newton}. In particular, it can be used to construct $M(u)$ from
$F'(u)$ by finite-dimensional approximation, which can then be exploited to
construct adaptive discretization schemes that optimize the contraction rate of
the algorithm in $V$. In the case of Finite Element analysis, our approach
delivers a multi-level Newton adaptive mesh refinement algorithm
\cite{Hohmann1994} that can be used straight-forward as another tool
complementing refinement strategies based on a posteriori error estimation (see,
e.g., \cite{Graetsch2005,Ainsworth2000,Becker2001}). 

Historically, Newton-type methods come in two flavors: Either, we do not solve
the linearized systems with operator $F'(u_k)$ and right-hand side $-F(u_k)$
exactly, which is usually known under the name \emph{inexact Newton} method
\cite{Dembo1982,Ypma1984}. Alternatively, we apply an approximation $M(u)$ of
the inverse of $F'(u)$ directly, an approach which is sometimes called
\emph{approximate Newton} method and which is the classical form of
Quasi-Newton methods \cite{Dennis1977}. As pointed out by Bank and Rose
\cite{Bank1981}, the two flavors are in fact different formulations of the same
class of methods. As mentioned above, our analysis is based on the approximate
Newton formulation in order to define the generalized Newton path
\eqref{eqn:genDavidenko}, but the control of the linearized residual in the
sense of inexact Newton methods emerges as the $\kappa$-condition
A\ref{ass:kappa}, which is identical to a choice of $\eta_k = \kappa < 1$ for
the residual forcing sequence $(\eta_k)$ first proposed in \cite{Dembo1982}. In
turn, the classical analysis of the local rates of convergence in
\cite{Dembo1982} carries over to our setting if $\kappa$ is allowed to be
reduced from iteration to iteration in the sense of a forcing sequence once we
are close to a solution.  This shall, however, not be the focus of this paper,
where we focus on the preasymptotic global convergence and are content with
locally linear convergence rates. 

Similar to \cite{Deuflhard1991}, we base our globalization approach on a
continuous curve, but we substitute here the exact Newton path (obtained with
the choice $M(u) = \smash[t]{F'(u)}^{-1}$ from \eqref{eqn:genDavidenko}) by a
generalized Newton path, which is allowed to have $M(u) \neq
\smash[t]{F'(u)}^{-1}.$ Following ideas of \cite{Hohmann1994}, we can use the
(contravariant) $\kappa$-condition A\ref{ass:kappa} to design a multilevel
Newton-type method for the construction of $M$ based on adaptive discretization.
This leads to adaptive discretizations solely based on balancing the
discretization residual with the nonlinear residual of the Newton-type method.
We remark that this convenient black-box approach might be inferior to more
involved schemes that balance the discretization errors and nonlinear errors
(instead of residuals) \cite{Rannacher2013} or exploit underlying (energy)
minimization properties \cite{Deuflhard1998} for particular problem classes.
Nonetheless, our method is the first for which convergence to the closest
solution in the sense of the generalized Newton flow \eqref{eqn:genDavidenko}
can be proven.

\paragraph{Contributions}
In this article, we extend the convergence analysis of backward step control for
\eqref{eqn:Newton} from the finite-dimensional to the Hilbert space setting. We
provide reasonable assumptions and convergence results for
\eqref{eqn:Newton} with backward step control. The main result is convergence to
a distinct solution characterized by the propagation of the initial guess by the
generalized Newton flow \eqref{eqn:genDavidenko} provided that no singularity of
the problem interferes. In addition, we prove an a priori bound on the nonlinear
reduction of the residual norm. The convergence theory can be exploited to
construct efficient numerical algorithms, which we discuss for the case of a
Krylov--Newton method and a Finite Element approximation. Both are based on the
optimization of the residual contraction constant, which yields in the latter
case an efficient adaptive mesh refinement strategy. The results are
demonstrated for the numerical solution of the Carrier equation and the minimal
surface equation.

\paragraph{Overview}
In section~\ref{sec:convergence}, we discuss the general assumptions, derive and
motivate the method of backward step control, provide reasons why implicit and
higher-order 
time stepping methods for \eqref{eqn:genDavidenko} are not advisable, provide
step size bounds, and establish the notion of generalized Newton paths, which
are the central ingredient for the analysis of local and global convergence of
backward step control. We then discuss extensions to Banach spaces and present
an algorithmic realization with a minimal working Matlab example code for the
solution of $\arctan(u) = 0$.
In section~\ref{sec:designM}, we exploit the
convergence analysis to construct two Newton-type methods, a Krylov--Newton
method and a method based on approximation-by-discretization. We apply these
methods in section~\ref{sec:ellipticPDE} to the class of nonlinear elliptic
boundary value problems and provide numerical results for the Carrier equation
and the minimum surface equation.

\paragraph{Notation} 

We denote the open ball of radius $r>0$ around $u\in U$ by $B(u,r)$ and the 
Laplace operator by $\Lap = \nabla \cdot \nabla$. As usual, we write $C^0$ for
the space of continuous functions, $H^1_0(\Omega)$ for the Sobolev space of
square integrable functions on a bounded domain $\Omega \subset \mathbb{R}^n$
that vanish at the boundary and admit square integrable derivatives, and
$H^{-1}(\Omega)$ for its dual space. The Euler number is denoted by $e =
\sum_{k=0}^{\infty} \frac{1}{k!}$.

\section{Convergence analysis}
\label{sec:convergence}

The overall structure of the backward step control convergence analysis in
Hilbert spaces is similar to the finite-dimensional case \cite{Potschka2016}.
The intricate interplay of the changes in the details, however, advises us to
present the convergence analysis in a self-contained fashion.

\subsection{Discussion of assumptions}

We start with the following definitions.
\begin{definition}
  The \emph{level function} $T: D \to \mathbb{R}$ is 
  $T(u) = \frac{1}{2} \norm{F(u)}_V^2.$
\end{definition}
\begin{definition}
  The \emph{level set of $u \in D$} is 
  $\smash[t]{\widetilde{\mathcal{T}}}(u) = \left\{ \bar{u} \in D \mid
  T(\bar{u}) \le T(u) \right\}.$
\end{definition}
\begin{definition}
  The \emph{path connected level set of $u \in D$} is
  \[
  \mathcal{T}(u) = \left\{ \bar{u} \in \widetilde{\mathcal{T}}(u)
  \mid \exists 
  c \in C^0\left([0, 1], \widetilde{\mathcal{T}}(u)\right)
  \text{with } c(0) = u, c(1) = \bar{u} \right\}.
  \]
\end{definition}
\begin{definition}
  For $r \in (1, \infty)$ the \emph{set of $r$-regular points} is
  \[
  \mathcal{R}_r = \left\{ u \in D \mid r^{-1} \norm{F(u)}_V <
  \norm{f(u)}_U < r \norm{F(u)}_V \right\}.
  \]
\end{definition}
\begin{definition}
  The \emph{set of $\infty$-regular points} is
  $\mathcal{R}_\infty = \bigcup_{r \in (1, \infty)} \mathcal{R}_r$.
\end{definition}
We remark that if $u \in D \setminus \mathcal{R}_\infty$, which means that $u
\not\in \mathcal{R}_r$ for all $r
\in (1, \infty)$, then $M(u)$ is either not bounded or does not admit a bounded
inverse \cite[\S I.6, Cor.~2, 3]{Yosida1978}. The contrary is, however,
not true: $M(u)$ may be unbounded or not admit a bounded inverse although $u \in
\mathcal{R}_r$ for some $r \in (1, \infty)$, because in the definition of
$\mathcal{R}_r$ only the action of $M(u)$ in direction $F(u)$ is of interest.

We require the following assumptions to hold true:
\begin{assumptions}
\item 
  \label{ass:validIni}
  There exists an $r \in (1, \infty)$ such that $u_0 \in \mathcal{R}_r$, and
  $\norm{F(u_0)}_V > 0$.
\item \label{ass:kappa}
  There exists a $\kappa < 1$ such that for all $u \in \mathcal{R}_r \cap
  \mathcal{T}(u_0)$
  \[
  \norm{F(u) - F'(u) f(u)}_V \le \kappa \norm{F(u)}_V.
  \]
\item \label{ass:omega}
  There exists an $\omega < \infty$ such that for all $u \in \mathcal{T}(u_0), t
  \in [0, 1]$
  \[
  \norm{\left[ F'(u) - F'(u-tf(u)) \right] f(u)}_V \le \omega t \norm{f(u)}_U
  \norm{F(u)}_V.
  \]
\item \label{ass:fLipschitz}
  There exists an $L < \infty$ such that for all $u, \bar{u} \in
  \mathcal{T}(u_0)$
  \[
  \norm{f(u) - f(\bar{u})}_U \le L \norm{u - \bar{u}}_U.
  \]
\item \label{ass:gamma}
  For all $\eta > 0$ there exist constants $\gamma$, $t_\gamma > 0$ such that
  for all $t \in [0, t_\gamma]$, $u \in \mathcal{R}_r \cap \mathcal{T}(u_0)$
  with $\norm{f(u)}_U > \eta$
  \[
  \norm{f(u-tf(u)) - f(u)}_U \ge \gamma t.
  \]
\end{assumptions}

The main difference in the assumptions here compared to the finite-dimensional
setting in \cite{Potschka2016} is the weakening of A\ref{ass:kappa} and
A\ref{ass:omega} from a formulation with matrices to a formulation which
requires the properties to hold only in the direction of the residual $F(u)$.
Thus, all requirements can be postulated without using norms for operators that
map between $U$ and $V$. 
Apart from the avoidance of operator norms, we had to replace all arguments
based on compactness of bounded sets by other means for the proofs in the
Hilbert space case.
The discussion of the assumptions in \cite{Potschka2016} still applies to a
large extent here: We require in A\ref{ass:validIni} that $u_0$ is an
$r$-regular point but not a solution. The central $\kappa$-condition
A\ref{ass:kappa} is a contravariant version of Bock's
covariant $\kappa$-condition \cite{Bock1987}
in the sense that it quantifies on the one hand the deviation of the inexact
increment $\delta u = -M(u) F(u)$ from the Newton increment $\delta
u^\mathrm{Newton} = -F(u)^{-1} F(u)$ in the $V$-norm
\[
  \norm{F'(u) \left[ \delta u^\mathrm{Newton} - \delta u \right]}_V
  = \norm{F(u) - F'(u)f(u)}_V
  \le \kappa \norm{F(u)}_V
\]
in comparison to Bock's covariant $\kappa^\mathrm{cov} < 1$ in the $U$-norm (see
also \cite[section~5.2]{Potschka2013})
\begin{equation}
  \label{eqn:kappa_Bock}
  \norm{M(u - f(u)) \left[ F(u) - F'(u)f(u) \right]}_U 
  \le \kappa^\mathrm{cov} \norm{f(u)}_U.
\end{equation}
On the other hand, $\kappa$ in A\ref{ass:kappa} characterizes the asymptotic
Q-linear convergence rate of the residual norms $\norm{F(u_k)}_V$, whereas
$\kappa^\mathrm{cov}$ in \eqref{eqn:kappa_Bock} characterizes the asymptotic
R-linear convergence rate of the error $\norm{u_k - u_\ast}_U$ if $u_\ast =
\lim_{k \to \infty} u_k$ (for a discussion of different
affine invariances see \cite{Deuflhard2006}). The $\omega$-condition
A\ref{ass:omega} measures a combination of the nonlinearity and the
well-posedness of \eqref{eqn:FOfxIsZero} because if $F'$ is Lipschitz continuous
with Lipschitz constant $L'$, then we obtain
\[
\norm{\left[ F'(u) - F'(u - tf(u)) \right] f(u)}_V \le L' t \norm{f(u)}_U^2
\]
and boundedness of $M(u)$ in direction $F(u)$ with constant $C$ implies
A\ref{ass:omega} with $\omega = CL'$. The Lipschitz condition
A\ref{ass:fLipschitz} is classical. The nonstandard assumption A\ref{ass:gamma}
follows, for instance, if $f$ is bi-Lipschitz with constant $\ell$
\[
\norm{f(u - tf(u)) - f(u)}_U \ge \ell t \norm{f(u)}_U
\]
with $\gamma = \eta \ell$ and $t_\gamma$ arbitrary.

\subsection{Backward step control}

Newton-type methods \eqref{eqn:Newton} are explicit Euler discretizations with
step sizes $t_k$ of the generalized Newton flow \eqref{eqn:genDavidenko}. Thus,
the convergence of Newton-type methods is strongly connected to the stability
problem of the explicit Euler method. Implict Euler, in contrast, has ideal
stability properties: It is an L-stable method (see, e.g., \cite{Hairer1996}).
Hence, in order to determine $t_k$, we consider the backward iterate
\[
\bar{u}_k(t_k) := u_{k+1} + t_k f(u_{k+1}) = u_k + t_k g(u_k, t_k) \quad
\text{with } g(u, t) := f(u - tf(u)) - f(u).
\]
The point $\bar{u}_k(t_k)$ is the starting point of a (stable) implicit Euler
step for \eqref{eqn:genDavidenko} that arrives exactly at $u_{k+1}$, the result
of a possibly unstable explicit Euler step starting from $u_k$. 
The idea of backward step control is based on a backward error argument: If a
small perturbation of the starting point $u_k$ can be found from which a stable
implicit Euler step arrives exactly at $u_{k+1}$, we can accept the step size.
We thus require that the distance 
between $u_k$ and $\bar{u}_k(t_k)$ is bounded by some fixed constant $H>0$
through the choice
\begin{equation}
  \label{eqn:BSC} \tag{BSC}
  t_k = \min \mathcal{B}_H(u_k) \quad \text{where }
  \mathcal{B}_H(u) = \left\{ t \in [0, 1] \mid H = t \norm{g(u,t)}_U
  \right\} \cup \{1\},
\end{equation}
which implies $\norm{\bar{u}_k(t_k) - u_k}_U \le H$ (with equality for $t_k <
1$) by continuity of $g$.

\subsection{Implicit and higher-order time stepping methods}

The question whether explicit Euler is really the best method to solve
\eqref{eqn:genDavidenko} arises naturally. We can answer this question
affirmatively for two reasons: First, all implicit methods have the drawback
that an approximated inverse of an operator involving derivatives of the
approximated inverse $M(u)$ would be required, e.g., in the case of the
implicit Euler method
\begin{align*}
  0 &= u_{k+1} + t_k f(u_{k+1}) - u_k,
  \intertext{with a local Newton corrector}
  u_{k+1}^{i+1} &= u_{k+1}^i - \left[ \eye_U + t_k f'(u_{k+1}^i) \right]^{-1}
  \left( u_{k+1}^i + t_k f(u_{k+1}^i) - u_k \right),
\end{align*}
which is not readily available and would require higher regularity of $M$ than
guaranteed by the assumptions above. 
Second, higher order methods would destroy the well-known locally quadratic
convergence of the Newton method, where $M(u) = (F'(u))^{-1}$.  This can be seen
from the homotopy formulation
\begin{equation}
  \label{eqn:uHomotopy}
  F(u(t)) - e^{-t} F(u_0) = 0,
\end{equation}
which we can differentiate with respect to $t$ to arrive exactly at
\eqref{eqn:genDavidenko} provided that $F'(u(t))$ stays invertible. 
Thus, the second order truncation error of explicit
Euler is required to obtain locally quadratic convergence, because higher
consistency orders would result in the locally linear convergence dictated by
\eqref{eqn:uHomotopy}.

We have not explored multi-step methods of order one further. The possible
outcome of this line of research is unfortunately unclear at present and exceeds
the scope of this paper.

\subsection{Step size bounds}

\begin{lemma}
  \label{lem:localFullSteps}
  If A\ref{ass:validIni} and A\ref{ass:fLipschitz} hold, then \eqref{eqn:BSC}
  delivers full steps $t_k = 1$ in the vicinity of a solution $u^\ast \in
  \mathcal{R}_r \cap \mathcal{T}(u_0)$.
\end{lemma}
\begin{proof}
  Let $u_k \in B(u^\ast, L^{-2} H)$. Hence, it holds for all $t \in [0, 1]$
  that
  \[
  t \norm{g(u_k, t)}_U 
  \overset{\text{A\ref{ass:fLipschitz}}}{\le} L t^2 \norm{f(u_k)}_U
  = L t^2 \norm{f(u_k) - f(u^\ast)}_U 
  \overset{\text{A\ref{ass:fLipschitz}}}{\le} L^2 t^2 \norm{u_k - u^\ast}_U < H.
  \]
  Thus, $\mathcal{B}_H(u_k) = \{1\}$ and $t_k = 1$ by virtue of \eqref{eqn:BSC}.
  \qed
\end{proof}

\begin{lemma}
  \label{lem:lowerStepsizeBound}
  If A\ref{ass:validIni} and A\ref{ass:fLipschitz} hold, then \eqref{eqn:BSC}
  generates for all $u_k \in \mathcal{R}_r \cap \mathcal{T}(u_0)$ step sizes
  that are either $t_k = 1$ or have the lower bounds
  \[
  t_k \ge \frac{\sqrt{H}}{\sqrt{L \norm{f(u_k)}_U}}
  > \frac{\sqrt{H}}{\sqrt{rL\norm{F(u_k)}_V}}
  \ge \frac{\sqrt{H}}{\sqrt{rL\norm{F(u_0)}_V}}.
  \]
\end{lemma}

\textit{Proof~}
  If $t_k < 1$, then
  \[
  t_k^2 \overset{\eqref{eqn:BSC}}{=} \frac{t_k H}{\norm{g(u_k, t_k)}_U}
  \overset{\text{A\ref{ass:fLipschitz}}}{\ge}
  \frac{H}{L \norm{f(u_k)}_U}
  > \frac{H}{rL\norm{F(u_k)}_V}
  \ge \frac{H}{rL\norm{F(u_0)}_V}. \qedhere
  \]

\begin{lemma}
  \label{lem:upperStepsizeBound}
  Let A\ref{ass:validIni} and A\ref{ass:gamma} hold and let $\bar{t}
  \in (0, 1)$ and $\eta > 0$. Then there exists an $\widebar{H} > 0$ such that
  for all $H \in (0, \widebar{H}]$ and $u \in \mathcal{R}_r \cap
  \mathcal{T}(u_0)$ with $\norm{f(u)}_U \ge \eta$ it holds that $\min
  \mathcal{B}_H(u) \le \bar{t}$.
\end{lemma}

\textit{Proof by contradiction~}
  We assume to the contrary that for all $\widebar{H} > 0$ there exists an $H
  \in (0, \widebar{H}]$ and a $u \in \mathcal{R}_r \cap \mathcal{T}(u_0)$
  satisfying $\norm{f(u)}_U \ge \eta$ and $\min \mathcal{B}_H(u) > \bar{t}$.
  Then, A\ref{ass:gamma} guarantees the existence of $\gamma$, $t_\gamma > 0$
  such that for $t := \min \{ t_\gamma, \bar{t} \} < \min \mathcal{B}_H(u)$ we
  obtain from \eqref{eqn:BSC} that
  \[
  \widebar{H} \ge H \ge t \norm{g(u, t)}_U \ge \gamma t^2 > 0.
  \]
  Because $\eta$ and thus $\gamma$ and $t$ are independent of $\widebar{H}$,
  we obtain a contradiction for $\widebar{H} \to 0$.
  \qedhere

\begin{lemma}
  \label{lem:dampedStepBound}
  Let A\ref{ass:validIni}, A\ref{ass:fLipschitz}, and A\ref{ass:gamma} hold and
  let $\eta > 0$. Then there exists an $\widebar{H} > 0$ such that for all $H
  \in (0, \widebar{H}]$ and $u \in \mathcal{R}_r \cap \mathcal{T}(u_0)$ it holds
  that
  \[
  \norm{f(u)}_U \min \mathcal{B}_H(u) \le \eta.
  \]
\end{lemma}
\begin{proof}
  We choose $\bar{t} \in (0, 1)$ sufficiently small so that it satisfies $r
  \norm{F(u_0)}_V \bar{t} \le \eta.$
  Then, Lemma \ref{lem:upperStepsizeBound} yields the existence of an
  $\widebar{H} > 0$ such that for all $H \in (0, \widebar{H}]$ and all $u \in
  \mathcal{R}_r \cap \mathcal{T}(u_0)$ with $\norm{f(u)}_U \ge \eta$ it holds
  that $\min \mathcal{B}_H(u) \le \bar{t}$. Hence,
  \[
  \norm{f(u)}_U \min \mathcal{B}_H(u)
  \overset{\text{A\ref{ass:validIni}}}{<} 
  r \norm{F(u)}_V \bar{t} \le r \norm{F(u_0)} \bar{t} \le \eta.
  \]
  For the remaining $u \in \mathcal{R}_r \cap \mathcal{T}(u_0)$ the assertion
  holds by virtue of $\norm{f(u)}_U < \eta$.
  \qed
\end{proof}

\subsection{Finite arclength of generalized Newton paths}

In the next step, we study the generalized Newton paths given by
\eqref{eqn:genDavidenko}. 
\begin{lemma}
  \label{lem:Fdescent}
  If A\ref{ass:validIni}, A\ref{ass:kappa}, A\ref{ass:fLipschitz} and $u_k \in
  \mathcal{R}_r \cap \mathcal{T}(u_0)$ hold, then there exists $\bar{t} > 0$
  such that \eqref{eqn:genDavidenko} has a unique local solution $u^k(t) \in
  \mathcal{R}_r \cap \mathcal{T}(u_k)$ for $t \in [0, \bar{t})$ which satisfies
  \[
  \norm{F(u^k(t))}_V \le e^{-(1-\kappa)t} \norm{F(u_k)}_V \quad 
  \text{for all } t \in [0, \bar{t}).
  \]
\end{lemma}
\begin{proof}
  The Picard--Lindel\"of theorem \cite[II.7, exercise 3]{Amann1990} yields with
  A\ref{ass:fLipschitz} the existence of a unique local solution $u^k(t)$ to
  \eqref{eqn:genDavidenko} in some neighborhood $(-\bar{t}, \bar{t})$ of $t =
  0$. Without loss of generality, $\bar{t} > 0$ is small enough to ensure
  $u^k(t) \in \mathcal{R}_r$ for $t \in [0, \bar{t})$ because $\mathcal{R}_r$ is
  open. For ease of notation, we abbreviate $u^k(t)$ by $u$. The Cauchy--Schwarz
  inequality and A\ref{ass:kappa} show that the level function is nonincreasing
  along this solution because
  \begin{align*}
    \frac{\ud}{\ud t} T(u) &= \inprod{F(u), F'(u) \frac{\ud u}{\ud t}}_V
    = -\inprod{F(u), F'(u) f(u)}_V\\
    &= -\norm{F(u)}_V^2 + \inprod{F(u), F(u) - F'(u) f(u)}_V\\
    &\le -\norm{F(u)}_V^2 + \kappa \norm{F(u)}_V^2
    = -2(1-\kappa) T(u) \le 0.
  \end{align*}
  Gronwall's inequality (see, e.g., \cite{Amann1990}) yields
  \[
  T(u^k(t)) \le e^{-2 (1-\kappa) t} T(u_k)
  \]
  and thus $u^k(t) \in \mathcal{T}(u_k)$ for $t \in [0, \bar{t})$. The assertion
  follows after multiplication by two and taking square roots.
  \qed
\end{proof}

We show in Theorem \ref{thm:arclength} below that the quantities in the
following definition are well-defined under suitable assumptions.
\begin{definition}
  \label{def:regularpart}
  For $r \in (1,\infty)$, we define the \emph{$r$-regular part $u_r^k$} of the
  generalized Newton path $u^k$ as the solution to the initial value problem
  \[
  \frac{\ud u_r^k}{\ud t}(t) = 
  \begin{cases}
    -f(u_r^k(t)) & \text{for } u_r^k(t) \in \mathcal{R}_r,\\
    0 & \text{otherwise},
  \end{cases}
  \quad \text{for } t \in [0, \infty),
  \quad \text{with } u_r^k(0) = u_k.
  \]
  We denote its limit by $u_k^\ast = \lim_{t \to \infty} u^k_r(t)$ and define
  $t_k^\ast = \inf \{t \in [0, \infty) \mid u_r^k(t) \not\in \mathcal{R}_r \}$
  with the usual convention that $\inf \varnothing = \infty$.
\end{definition}

\begin{theorem}
  \label{thm:arclength}
  Let A\ref{ass:validIni}, A\ref{ass:kappa}, and A\ref{ass:fLipschitz} hold. If
  $u_k \in \mathcal{R}_r \cap \mathcal{T}(u_0)$, then the $r$-regular part of
  the generalized Newton path exists uniquely and has a finite arclength
  $\ell(u_r^k)$
  satisfying
  \[
  \norm{u_k - u_k^\ast}_U \le
  \ell(u_r^k) < \frac{r}{1 - \kappa} \norm{F(u_k)}_V 
  < \frac{r^2}{1 - \kappa} \norm{f(u_k)}_U.
  \]
  If $u_k^\ast \in \mathcal{R}_r$, then $F(u_k^\ast) = 0$.
\end{theorem}

\textit{Proof~}
  The unique local solution of Lemma \ref{lem:Fdescent} can be extended uniquely
  in $\mathcal{T}(u_k)$ by repeated application of the Picard--Lindel\"of
  theorem either until $u^k(t) \not \in \mathcal{R}_r$ for some $t = t_k^\ast$
  or to the whole interval $t \in [0, \infty)$. In the first case, the
  $r$-regular part is uniquely determined by $u_r^k(t) = u_r^k(t_k^\ast)$ for
  all $t \ge t_k^\ast$. We can now use the definition of $\mathcal{R}_r$ and
  Lemma \ref{lem:Fdescent} in order to show
  \begin{align*}
    \ell(u_r^k) &= \int_{0}^{\infty} \norm{\frac{\ud u_r^k}{\ud t}(t)}_U \ud t
    = \int_{0}^{t_k^\ast} \norm{f(u^k(t))}_U \ud t
    < r \int_{0}^{t_k^\ast} \norm{F(u^k(t))}_V \ud t \\
    &\le r \int_{0}^{\infty} e^{-(1-\kappa)t} \ud t \norm{F(u_k)}_V
    = \frac{r}{1-\kappa} \norm{F(u_k)}_V < \frac{r^2}{1-\kappa}
    \norm{f(u_k)}_U. 
  \end{align*}
  We obtain the lower arclength bound by noting that the shortest path between
  $u_k$ and $u_k^\ast$ has arclength $\norm{u_k - u_k^\ast}_U$. If $u_k^\ast \in
  \mathcal{R}_r$, then $t_k^\ast = \infty$ and Lemma \ref{lem:Fdescent} reveals
  \[
  \norm{F(u_k^\ast)}_V \le \lim_{t \to \infty} e^{-(1-\kappa)t}\norm{F(u_k)}_V =
  0.
  \qedhere
  \]

\subsection{Local convergence} 

We use the next lemma to prove discrete descent of the residual norm.
\begin{lemma}
  \label{lem:kappaScaled}
  A\ref{ass:kappa} holds if and only if for all $u \in \mathcal{R}_r \cap
  \mathcal{T}(u_0)$ and $t \in [0, 1]$
  \[
  \norm{F(u) - t F'(u) f(u)}_V 
  \le \left[ 1 - (1-\kappa) t \right] \norm{F(u)}_V.
  \]
\end{lemma}
\textit{Proof~}
  As in \cite{Potschka2016}, the nontrivial direction of the proof follows from
  the convexity of the functional $\varphi(t) = \norm{F(u) - t F'(u) f(u)}_V$
  and A\ref{ass:kappa} according to
  \[
  \varphi(t) \le (1-t)\varphi(0) + t \varphi(1) 
  \le \left[ (1-t) + \kappa t \right] \norm{F(u)}_V.
  \qedhere
  \]

\begin{lemma}
  \label{lem:Fdecrease}
  Let A\ref{ass:validIni}, A\ref{ass:kappa} and A\ref{ass:omega} hold. If $u_k
  \in \mathcal{R}_r \cap \mathcal{T}(u_0)$, then
  \[
  \norm{F(u_{k+1})}_V \le \left[ 1 - (1-\kappa) t_k + \frac{\omega}{2}
  \norm{f(u_k)}_U t_k^2 \right] \norm{F(u_k)}_V.
  \]
  Furthermore, if there exists a $\theta < 1$ such that the step size sequence
  satisfies
  \begin{equation*}
  \omega t_k \norm{f(u_k)}_U \le 2 \theta (1 - \kappa),
  \end{equation*}
  then
  \[
  \norm{F(u_{k+1})}_V \le \left[ 1 - (1-\theta) (1-\kappa) t_k \right]
  \norm{F(u_k)}_V.
  \]
\end{lemma}
\begin{proof}
  Because $F'$ is continuously G\^ateaux differentiable, we can apply
  \cite[3.2.2.~Thm. and 2.1.4.~Thm.]{Hamilton1982} to show that for $u \in D$
  and $\delta u \in U$
  \[
    F(u + t_k \delta u) - F(u) =
    \int_{0}^{1} F'(u + \tau t_k \delta u) t_k \delta u \ud \tau
    \overset{\tau = \frac{t}{t_k}}{=} \int_{0}^{t_k} F'(u + t \delta u) \delta
    u \ud t.
  \]
  Using Lemma \ref{lem:kappaScaled} and A\ref{ass:omega} we obtain the first
  assertion from
  \begin{align*}
    & \quad\, \norm{F(u_{k+1})}
    = \norm{F(u_k) - \int_{0}^{t_k} F'(u_k - \tau f(u_k)) f(u_k) \ud \tau}_V\\
    &= \norm{F(u_k) - t_k F'(u_k) f(u_k) + \int_{0}^{t_k} \left[ F'(u_k) -
    F'(u_k - \tau f(u_k)) \right] f(u_k) \ud \tau}_V\\
    &\le \norm{F(u_k) - t_k F'(u_k) f(u_k)}_V
    + \int_{0}^{t_k} \norm{\left[ F'(u_k) - F'(u_k - \tau f(u_k)) \right]
    f(u_k)}_V \ud \tau\\
    &\le \left[ 1 - (1 - \kappa) t_k + \frac{\omega}{2} \norm{f(u_k)}_U t_k^2
    \right] \norm{F(u_k)}_V.
  \end{align*}
  The second assertion follows immediately.
  \qed
\end{proof}

We can now state a local convergence theorem.
\begin{theorem}
  \label{thm:localConvergence}
  Let A\ref{ass:validIni}, A\ref{ass:kappa}, and A\ref{ass:omega} hold. If there
  exists a $\bar{t} \in (0,1)$ such that 
  $t_k \ge \bar{t}$ for all $k \in \mathbb{N}$
  and if there exists a $\theta \in (0, 1)$ such that for some $k \in
  \mathbb{N}$ the iterate $u_k$ satisfies
  \[
  \mathcal{T}(u_k) \subseteq \mathcal{R}_r
  \quad \text{and} \quad
  \omega r \norm{F(u_k)}_V \le 2 \theta (1 - \kappa),
  \]
  then $(u_k)_{k \in \mathbb{N}}$ converges to some point $u^\ast \in
  \mathcal{T}(u_k)$ with $F(u^\ast) = 0$.
\end{theorem}
\begin{proof}
  Because $t_k \in [0, 1]$, we have
  \[
  \omega t_k \norm{f(u_k)}_U \le \omega \norm{f(u_k)}_U
  \overset{\text{A\ref{ass:validIni}}}{\le} \omega r \norm{F(u_k)}_V
  \le 2 \theta (1 - \kappa).
  \]
  Hence, repeated application of Lemma \ref{lem:Fdecrease} yields for all $j \in
  \mathbb{N}$
  \begin{equation}
    \label{eqn:geomFConv}
    \norm{F(u_{k+j})}_V 
    \le q^j \norm{F(u_k)}_V \quad \text{with }
    q := 1 - (1-\theta)(1-\kappa) \bar{t}.
  \end{equation}
  Because $q < 1$, $\norm{F(u_k)}_V$
  converges geometrically. In addition, we obtain that $(u_k)_{k \in
  \mathbb{N}}$ is a Cauchy sequence by virtue of
  \begin{align*}
    \norm{u_k - u_{k+j}}_U &\le \sum_{i=0}^{j-1} \norm{u_{k+i} - u_{k+i-1}}_U
    = \sum_{i=0}^{j-1} t_{k+i} \norm{f(u_{k+i})}_U\\
    &\le r \sum_{i=0}^{j-1} \norm{F(u_{k+i})}_V 
    \le r \norm{F(u_{k})}_V \sum_{i=0}^{\infty} q^i 
    = \frac{r}{1-q} \norm{F(u_k)} \overset{k \to \infty}{\longrightarrow} 0.
  \end{align*}
  Thus, $(u_{k})_{k \in \mathbb{N}}$ converges to some $u^\ast \in \mathcal{T}(u_k)
  \subseteq \mathcal{R}_r$ and \eqref{eqn:geomFConv} implies $F(u^\ast) = 0$.
  \qed
\end{proof}

For the rate of convergence, we obtain the following result:
\begin{lemma}
  \label{lem:linearconvergence}
  Under the assumptions of Theorem \ref{thm:localConvergence}, $\norm{F(u_k)}_V$
  converges linearly with asymptotic linear convergence rate $\kappa < 1$.
\end{lemma}

\begin{proof}
  Because $u_k \in \mathcal{R}_r$, it follows that $\norm{f(u_k)}_U \le r
  \norm{F(u_k)}_V \to 0$. Hence, there is a sequence $(\theta_K)_{K \in
  \mathbb{N}}$ with $\theta_K \in [0, 1)$ and $\theta_K \to 0$ such that
  \[
  \omega t_k \norm{f(u_k)}_U \le 2 \theta_K (1 - \kappa) \quad \text{for all } k
  \ge K.
  \]
  Lemma \ref{lem:localFullSteps} and repeated application of Lemma
  \ref{lem:Fdecrease} then deliver 
  \[
  \norm{F(u_{K+1})}_V \le \left[ 1 - (1-\theta_K)(1-\kappa) \right]
  \norm{F(u_K)}_V,
  \]
  where $1 - (1-\theta_K)(1-\kappa) \to \kappa$ as $K \to \infty$.
  \qed
\end{proof}

In order to obtain methods with guaranteed superlinear or quadratic local
convergence, it is necessary to appropriately drive $\kappa$ to zero as $F(u_k)
\to 0$ as exhaustively described by the means of forcing sequences $\eta_k =
\kappa$ in \cite{Dembo1982}.

\subsection{Global convergence}

The following lemmas are required for the main theorem, which assures
convergence of the iterates to $u_0^\ast$.
As a prerequisite, we prove that every neighborhood of an isolated zero $u^\ast$
of $F$ contains a path connected level set that contains a neighborhood of
$u^\ast$.
\begin{lemma}
  \label{lem:boundedLevelSets}
  Let A\ref{ass:validIni}, A\ref{ass:kappa}, and A\ref{ass:fLipschitz} hold. If
  there exist $\varepsilon > 0$ and $u^\ast \in D$ such that $u^\ast$ is the
  only zero of $F$ on $B(u^\ast, \varepsilon) \subseteq \mathcal{R}_r \cap
  \mathcal{T}(u_0),$ then there exists an $\tilde{\varepsilon} > 0$ with
  \[
  \bigcup_{u \in B(u^\ast, \tilde{\varepsilon})} \mathcal{T}(u) 
  \subseteq B(u^\ast, \varepsilon).
  \]
\end{lemma}

\textit{Proof by contradiction~}
  We assume to the contrary that there exists a sequence $(u_n)_{n \in
  \mathbb{N}}$ with $\norm{u_n - u^*}_U < \frac{\varepsilon}{2n}$ and 
  $\mathcal{T}(u_n) \not\subseteq B(u^*, \varepsilon)$. 
  Hence, there exists a sequence $(\tilde{v}_n)_{n \in \mathbb{N}}$ with
  $\tilde{v}_n \in \mathcal{T}(u_n)$ and $\norm{\tilde{v}_n - u^\ast}_U \ge
  \varepsilon$.  Because $\mathcal{T}(u_n)$ is path connected, there exist
  continuous functions $c_n: [0, 1] \to \mathcal{T}(u_n)$ with $c_n(0) = u_n$
  and $c_n(1) = \tilde{v}_n$. Because $\norm{c_n(0) - u^*}_U <
  \frac{\varepsilon}{2}$ and $\norm{c_n(1) - u^*}_U \ge \varepsilon$, the
  intermediate value theorem yields the existence of $v_n = c_n(\tau_n) \in
  \mathcal{T}(u_n)$ for some $\tau_n \in [0, 1]$ satisfying 
  \begin{equation}
    \label{eqn:annulus}
    \norm{v_n - u^*}_U = \frac{\varepsilon}{2}.
  \end{equation}
  By Theorem \ref{thm:arclength}, we obtain for the distance to the limit
  $v_n^\ast$ of the $r$-regular part of the generalized Newton path emanating
  from $v_n$ that
  \begin{align*}
    \norm{v_n - v_n^\ast}_U 
    &\le \ell(v_r^n) < \frac{r}{1-\kappa} \norm{F(v_n)}_V
    \le \frac{r}{1-\kappa} \norm{F(u_n)}_V
    < \frac{r^2}{1-\kappa} \norm{f(u_n)}_U\\
    &= \frac{r^2}{1-\kappa} \norm{f(u_n) - f(u^*)}_U
    \le \frac{r^2L}{1-\kappa} \norm{u_n - u^*}_U <
    \frac{r^2L\varepsilon}{2(1-\kappa)n} \to 0,
  \end{align*}
  which implies for some sufficiently large $n$ that $v^*_n \in B(u^*,
  \varepsilon) \subseteq \mathcal{R}_r$ and thus $F(v^*_n) = 0$.  By
  \eqref{eqn:annulus} we get $v_n^* \neq u^*$ in contradiction to the uniqueness
  of $u^*$.
  \qedhere

We also need a bound on the deviation of two neighboring generalized Newton
paths emanating from $u_k$ and $u_{k+1}$.
\begin{lemma}
  \label{lem:oneStepError}
  Let A\ref{ass:fLipschitz} hold. If $u^k(t_k + \tau)$, $u^{k+1}(\tau) \in
  \mathcal{T}(u_k)$ for all $\tau \in [0, t]$, then
  \[
  \norm{u^k(t_k + t) - u^{k+1}(t)}_U 
  \le \frac{1}{2} \norm{f(u_k)}_U L e^{L(t_k+t)} t_k^2.
  \]
\end{lemma}
\begin{proof}
  We use the integral form of the Gronwall inequality as in
  \cite[Lem.~8.5]{Potschka2016}.
  \qed
\end{proof}

In order to prove an a priori bound on the decrease of the nonlinear residual
for \eqref{eqn:BSC}, we need the following Lemma.

\begin{lemma}
  \label{lem:sqrtBound}
  Let $h > 0$. If a sequence $(a_k)_{k \in \mathbb{N}}$ of nonnegative 
  numbers satisfies $a_{k+1}^2 \le a_k^2 - 2 h a_k$ for all $k \in \mathbb{N},$
  then $a_k \le \max \{a_0 - k h, 0\}$ for all $k \in \mathbb{N}.$
\end{lemma}
\begin{proof}
  If $a_k \le 2 h$, it follows that $a_{k+1} = 0$ by virtue of
  $0 \le a_{k+1}^2 \le a_k (a_k - 2h) \le 0.$
  If $a_k > 2 h$, we immediately obtain
  $a_{k+1}^2 \le a_k^2 - 2 h a_k + h^2 = (a_k - h)^2.$
  Taking the square root completes the proof.  \qed
\end{proof}

We can now prove the main theorem of this article.

\begin{figure}[tb]
  \sidecaption
    \begin{tikzpicture}
      \draw[color=black!50] (0,0) circle (1);
      \draw[color=black!50] (-0.9,0) circle (0.7);
      \draw (0,-2) arc (-90:90:2);
      \draw[color=black!50] (0,2) arc (90:270:2);
      \draw (0,2) -- (-1,2);
      \draw (-1,2) arc (90:150:5);
      \draw (0,-2) -- (-1,-2);
      \draw (-1,-2) arc (90:150:1);
      \draw[thick] (0,0) -- (-1,0);
      \draw[thick] (-1,0) arc (90:150:3);
      \draw[thick] (-1,-0.1) -- (-1,0.1);
      \draw[thick] (-0.8,-0.1) -- (-0.8,0.1);
      \draw (-0.8,-0.7) arc (-90:90:0.7);
      \draw (-0.8,-0.7) -- (-1,-0.7);
      \draw (-0.8,0.7) -- (-1,0.7);
      \draw (-1,0.7) arc(90:150:3.7);
      \draw (-1,-0.7) arc(90:150:2.3);
      \fill (0,0) circle (2pt) node[right] {$u_0^\ast$};
      \draw[->,color=black!50] (-0.9,-1.2)  
      -- (-0.9,0) node {\tikz \fill (0,0) circle (1pt);};
      \node at (0,-1.4) {$u^0\left( \sum_{i=0}^{k-1} t_i \right)$};
      \node[above] at (0,1) {$B(u_0^\ast, \frac{\varepsilon}{2})$};
      \node[rotate=30] at (-3.8,1.5) {$\mathcal{N}_\varepsilon$};
      \node[rotate=30] at (-3,0.5) {$\mathcal{N}_{\tilde{\varepsilon}}$};
      \node[rotate=30] at (-2.5,-0.1) {$u^0(t)$};
    \end{tikzpicture}
  \caption{The idea of the proof of Theorem \ref{thm:convergenceBSC} is based on
  steering an iterate $u_k$ into the $\tilde{\varepsilon}$-ball around
  $u^0\smash[tb]{\left(\sum_{i=0}^{k-1} t_i\right)}$, which is fully contained
  in the $\varepsilon$-ball around the solution $u_0^\ast$, the region of local
  convergence. The short black vertical dashes on the generalized Newton path
  $u^0(t)$ mark the points $u^0(T_\ast - 1)$ and $u^0(T_\ast)$.}
  \label{fig:convergenceBSC}
\end{figure}
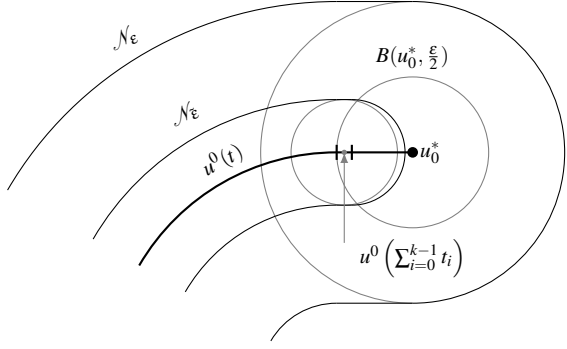

\begin{theorem}
  \label{thm:convergenceBSC}
  Let A\ref{ass:validIni}, A\ref{ass:kappa}, A\ref{ass:omega},
  A\ref{ass:fLipschitz}, and A\ref{ass:gamma} hold. If there exists an
  $\varepsilon > 0$ such that
  \[
  \mathcal{N}_\varepsilon := \bigcup_{t \in [0, \infty)} B(u^0(t), \varepsilon)
  \subseteq \mathcal{R}_r
  \]
  and $u_0^\ast$ is an isolated zero of $F$, then there exist constants
  $\widebar{H} >0$, $K < \infty$, and $c > 0$ such that the following statements
  hold true:
  \begin{enumerate}
    \item For all $H \in (0, \widebar{H}]$ iteration \eqref{eqn:Newton} with
      step size selection \eqref{eqn:BSC} converges to $u_0^\ast$.

    \item After $\smash[t]{\bar{k}} = \lceil K / \smash[t]{\sqrt{H}} \rceil$
      steps, $u_{\bar{k}}$ lies in the region of local full step convergence.

    \item The sequence of residual norms $\norm{F(u_k)}_V$ decreases
      geometrically and satisfies the additional a priori bound
      \[
        \sqrt{\norm{F(u_k)}_V}
        \le \sqrt{\norm{F(u_0)}_V} - k c \sqrt{H}
      \]
      for all $k \in \{ i \in \mathbb{N} \mid t_j < 1, j = 0, \dotsc, i \}$,
      which includes $k = \smash[t]{\bar{k}}.$
\end{enumerate}
\end{theorem}
\begin{proof}
  Because $u_0^\ast$ is an isolated zero of $F$, we can assume without loss
  of generality that $\varepsilon > 0$ was chosen small enough such that
  $u_0^\ast$ is the only zero of $F$ in $\mathcal{N}_\varepsilon$. 
  Let now $\theta \in (0, 1)$.  We further assume without loss of generality
  $\varepsilon$ to be sufficiently small to satisfy
  $\omega r^2 L \varepsilon \le 2 \theta (1-\kappa)$.
  It follows with A\ref{ass:validIni} and A\ref{ass:fLipschitz} that for all $u
  \in B(u_0^\ast, \varepsilon)$
  \begin{align*}
    \omega r \norm{F(u)}_V 
    &< \omega r^2 \norm{f(u)}_U
    = \omega r^2 \norm{f(u) - f(u_0^\ast)}_U\\
    &\le \omega r^2 L \norm{u - u_0^\ast}_U
    < \omega r^2 L \varepsilon
    \le 2 \theta (1 - \kappa).
  \end{align*}
  Thus, Theorem \ref{thm:localConvergence} establishes that $B(u_0^\ast,
  \varepsilon)$ is contained in the region of local full step convergence.
  By possibly reducing $\varepsilon$ further according to Lemma
  \ref{lem:boundedLevelSets}, we can guarantee that if $u_{k} \in
  B(u_0^\ast, \varepsilon)$ for some $k \in \mathbb{N}$, then $(u_k)_k$
  converges to the unique zero $u_0^\ast$ in $\mathcal{N}_\varepsilon.$ For the
  first statement, it now remains to show that $u_k \in B(u_0^\ast,
  \varepsilon)$ for some $k \in \mathbb{N}.$

  To this end, we choose $T_\ast < \infty$ such that
  \[
  \norm{u^0(t) - u_0^\ast}_U \le \frac{\varepsilon}{2} \quad \text{for all } t
  \ge T_\ast - 1.
  \]
  Because $u_0^\ast$ is the only zero of $F$ on $\mathcal{N}_\varepsilon$, we
  can choose an $\tilde{\varepsilon} \in (0, \frac{\varepsilon}{2})$ such that
  there is an $\eta > 0$ satisfying
  \[
  \norm{f(u)}_U > \eta \quad \text{for all } 
  u \in \mathcal{T}(u_0) \cap \mathcal{N}_{\tilde{\varepsilon}}, \text{ where }
  \mathcal{N}_{\tilde{\varepsilon}} := \bigcup_{t \in [0, T_\ast]} B(u^0(t),
  \tilde{\varepsilon}).
  \]
  From A\ref{ass:gamma} we obtain the existence of constants $\gamma > 0$ and
  $t_\gamma \in (0, 1)$ such that
  \begin{equation}
    \label{eqn:lbg}
    \norm{g(u, t)}_U \ge \gamma t \quad \text{for all } t \in [0, t_\gamma]
    \text{ and } u \in \mathcal{N}_{\tilde{\varepsilon}}.
  \end{equation}
  We can then use $\bar{t} = t_\gamma$ in Lemma \ref{lem:upperStepsizeBound} to
  obtain a constant $\widebar{H} > 0$ such that
  \begin{equation}
    \label{eqn:ubtk}
    \min \mathcal{B}_H(u) \le t_\gamma < 1 \quad \text{for all } H \in (0,
    \widebar{H}] \text{ and } u \in \mathcal{N}_{\tilde{\varepsilon}}.
  \end{equation}
  In anticipation of a later argument in the proof, we can assume without loss
  of generality that $\widebar{H}$ is sufficiently small to satisfy
  \begin{equation}
    \label{eqn:smallHbar}
    T_\ast e^{L T_\ast}\left( rL \norm{F(u_0)}_V \right)^{\frac{3}{2}}
    \smash[t]{\widebar{H}}^{\frac{1}{2}} \le 2 \gamma \tilde{\varepsilon}.
  \end{equation}
  Combining the inequalities \eqref{eqn:lbg} and \eqref{eqn:ubtk} with
  \eqref{eqn:BSC}, we see that
  \begin{equation}
    \label{eqn:lbtk2}
    \gamma t_k^2 \le t_k \norm{g(u_k, t_k)}_U \le H \quad \text{for all } H \in
    (0, \widebar{H}] \text{ and } u_k \in \mathcal{N}_{\tilde{\varepsilon}}.
  \end{equation}
  We now choose an $H$-dependent $\bar{k} \in \mathbb{N}$ that satisfies
  \[
  T_\ast - 1 \le \sum_{i=0}^{\bar{k}-1} t_i \le T_\ast
  \]
  and show by induction that $u_k \in \mathcal{T}(u_0) \cap
  \mathcal{N}_{\tilde{\varepsilon}}$ for all $k \le \bar{k}$, which clearly
  holds true for $k=0$ because $u_0 = u^0(0) \in B(u_0, \tilde{\varepsilon}).$
  We can now assume inductively that $u_i \in \mathcal{T}(u_0) \cap
  \mathcal{N}_{\tilde{\varepsilon}}$ for all $i \le k-1$ with $k \le \bar{k}$ in
  order to show $u_k \in  \mathcal{T}(u_0) \cap
  \mathcal{N}_{\tilde{\varepsilon}}.$
  Because $\mathcal{N}_{\tilde{\varepsilon}} \subseteq \mathcal{R}_r$, it
  follows from \eqref{eqn:ubtk} that $t_i < 1$. Hence, Lemma
  \ref{lem:lowerStepsizeBound} yields
  \[
  T_\ast \ge
  \sum_{i=0}^{k-1} t_i \ge k \frac{\sqrt{H}}{\sqrt{r L \norm{F(u_0)}_V}},
  \]
  which implies the bound
  \begin{equation}
    \label{eqn:ubkbar}
    k \le \frac{T_\ast \sqrt{rL\norm{F(u_0)}_V}}{\sqrt{H}}.
  \end{equation}
  We then obtain by a telescope argument, Lemma \ref{lem:oneStepError},
  A\ref{ass:validIni}, \eqref{eqn:lbtk2}, \eqref{eqn:ubkbar}, and
  \eqref{eqn:smallHbar} that
  \begin{align*}
    \norm{u^0 \left( \sum_{i=0}^{k-1} t_i \right) - u_k}_U
    &\le \sum_{j=0}^{k-1} \norm{u^j \left( \sum_{i=j}^{k-1} t_i \right) -
    u^{j+1} \left( \sum_{i=j+1}^{k-1} t_i \right)}_U\\
    &\le
    \sum_{j=0}^{k-1} \frac{1}{2} \norm{f(u_{j})}_U L e^{L T_\ast} t_j^2
    < \frac{rL e^{LT_\ast} \norm{F(u_0)}_V}{2 \gamma} H k\\
    &\le \frac{T_\ast e^{LT_\ast} \left( rL \norm{F(u_0)}_V
    \right)^{\frac{3}{2}} }{2 \gamma} \sqrt{H} \le \tilde{\varepsilon}.
  \end{align*}
  Because $\sum_{i=0}^{k-1} t_i \le T_\ast$, we have established by induction
  that $u_k \in \mathcal{T}(u_0) \cap \mathcal{N}_{\tilde{\varepsilon}}$ for all
  $k \le \bar{k}$. Finally, $u_{\bar{k}} \in B(u_0^\ast, \varepsilon)$ by virtue
  of
  \[
  \norm{u_0^\ast - u_{\bar{k}}}_U
  \le \norm{ u_0^\ast - u^0 \left( \sum_{i=0}^{\bar{k}-1} t_i \right)}_U
  + \norm{ u^0 \left( \sum_{i=0}^{\bar{k}-1} t_i \right) - u_{\bar{k}}}_U
  < \frac{\varepsilon}{2} + \tilde{\varepsilon} < \varepsilon.
  \]
  This proves the first statement.
  The second assertion follows from \eqref{eqn:ubkbar} by choosing $K = T_\ast
  \sqrt{r L \norm{F(u_0)}_V}$ for $k = \bar{k}$.
  For the third statement, we obtain from the lower step size bound of Lemma
  \ref{lem:lowerStepsizeBound} and Lemma \ref{lem:Fdecrease} that for all $k$
  with $t_k < 1$ it holds that
  \begin{gather*}
    \norm{F(u_{k+1})}_V \le \left[ 1 - (1-\theta)(1-\kappa)t_k \right]
    \norm{F(u_k)}_V
    \le \norm{F(u_k)}_V - 2 c \sqrt{H \norm{F(u_k)}_V},\\
    \text{with } c = \frac{(1-\theta)(1-\kappa)}{2 \sqrt{rL}} > 0.
  \end{gather*}
  Hence, the residual norm sequence converges geometrically and
  Lemma \ref{lem:sqrtBound} yields the a priori bound with $a_k =
  \sqrt{\norm{F(u_k)}_V}$ and $h = c \sqrt{H}$. Because we established $u_k \in
  \mathcal{N}_{\tilde{\varepsilon}}$ for all $k \le \smash[t]{\bar{k}}$ in
  the proof of the first statement, we obtain $t_k < 1$ for all $k \le
  \smash[t]{\bar{k}}$ from \eqref{eqn:ubtk}, which completes the proof.
  \qed
\end{proof}

\subsection{A note on generalizations to Banach spaces}
\label{sec:Banach}

The only step in the convergence proof that exploits the Hilbert space structure
of $V$ is Lemma \ref{lem:Fdescent}. All remaining steps can be carried out even
if $V$ is only a Banach space.  The general approach here is to modify the used
level function and to require an additional assumption akin to condition
A\ref{ass:kappa}.

In this section, we shortly present the
necessary modifications for the special case of the Lebesgue space $V =
L^p(\Omega)$ of real-valued $p$-integrable functions with $ 2 < p <
\infty$ and $\Omega \subset \mathbb{R}^n$. First, we need to consider a
different level function 
\[
T_p(u) = \frac{1}{p} \norm{F(u)}_V^p 
= \frac{1}{p} \int_{\Omega} \abs{F(u)(s)}^p \ud s,
\]
with analogously defined level sets $\mathcal{T}_p(u)$.
In addition to A\ref{ass:kappa}, we require pointwise that
\[
\left[F(u) - F'(u)f(u)\right](s) \le \kappa \abs{F(u)(s)} \quad \text{for almost all } s
\in \Omega \text{ and all } u \in \mathcal{R}_r \cap \mathcal{T}_p(u_0).
\]
Standard arguments on the differentiability of $p$-norms (see, e.g., \cite[Thm.
2.6]{Lieb2001}) then deliver with the abbreviation $u = u^k(t)$ that
\begin{align*}
  & \quad\, \frac{\ud}{\ud t} T_p(u^k(t))
  = -\int_{\Omega} \abs{F(u)(s)}^{p-2} \left[ F(u)(s) \cdot ( F'(u) f(u)) (s)
  \right] \ud s\\
  &= -\int_{\Omega} \abs{F(u)(s)}^p \ud s 
  + \int_{\Omega} \abs{F(u)(s)}^{p-2} \left[ F(u)(s) \cdot (F(u) - F'(u) f(u))
  (s) \right] \ud s\\
  &\le -p T_p(u)
  + \int_{\Omega} \abs{F(u)(s)}^{p-1} (F(u) - F'(u) f(u)) (s) \ud s\\
  &\le -p(1-\kappa) T_p(u) \le 0.
\end{align*}
After application of Gronwall's inequality, multiplication by $p$ and taking the
$p$-th root, we establish the result of Lemma \ref{lem:Fdescent}
\[
\norm{F(u^k(t))}_V \le e^{-(1-\kappa)t} \norm{F(u_k)}_V.
\]

\subsection{Algorithmic realization}
\label{sec:realization}

The algorithmic realization of \eqref{eqn:BSC} can be carried over verbatim from
the finite-dimensional setting laid out in \cite[section~10]{Potschka2016} with the
use of $\norm{.}_U$ for all occuring norms. As in \cite{Potschka2016}, we do not
use monotone iterations \cite{Galantai2014} for numerical computations here but
use the bisection procedure with exponentially smoothed step size prediction.
For convenience, we sketch it again: In order to determine $t_k$ from
\eqref{eqn:BSC}, we approximately compute a zero of the Lipschitz continuous
scalar function $t \mapsto t \norm{g(u_k, t)}_U - H$ by a bracketing procedure.
Numerically, we are content with a $t_k$ that satisfies
\begin{equation}
  \label{eqn:HlHu}
  H'_k := t_k \norm{g(u_k, t_k)}_U \in [H^\mathrm{l}, H^\mathrm{u}] 
  \quad \text{or} \quad
  t_k \approx 1 \text{ and } H'_k < H^\mathrm{l},
\end{equation}
where $H^\mathrm{l} < H$ and $H^\mathrm{u} > H$ are close to $H$. 

\lstdefinestyle{mymatlab}{%
  basicstyle=\footnotesize,
  basewidth=0.5em,
  showstringspaces=false,
  numbers=left,
  xleftmargin=2.5em,
  framexleftmargin=2em}
\begin{table}[bt]
  \caption{Example Matlab code for solving $\arctan(u) = 0$ from $u_0 = 2$ with
  backward step control.}
  \label{alg:example_atan}
\begin{lstlisting}[language=Matlab,style=mymatlab]
f = @(u) (u.^2 + 1) * atan(u);
u = 2; du = -f(u); H = 0.8; Hprime = H; t = 1;
fprintf('%3s %7s %9s %9s %9s %9s\n', 'k', 't', 'u', 'du', 'dup', 'Hprime')
for k = 0:5
    t = min(1, t * (0.8 + 0.2*H/Hprime));
    tl = 0; tu = 1;
    while 1 % bisection
        up = u + t * du; dup = -f(up);
        Hprime = t * norm(dup - du);
        fprintf('%3d %7.4f %9.1e %9.1e %9.1e %9.1e', k, t, u, du, dup, Hprime)
        if Hprime < 0.1 * H && t < 0.999
            tl = t; t = 0.5 * (tu + t); fprintf('  increase t\n')
        elseif Hprime > 2.0 * H
            tu = t; t = 0.5 * (tl + t); fprintf('  decrease t\n')
        else fprintf('  accept t\n'), break, end
    end
    u = up; du = dup;
end

\end{lstlisting}
\end{table}

For illustration purposes, we provide in Table \ref{alg:example_atan}
example code in Matlab, which computes for the real-valued function $F(u) =
\arctan(u)$ the iterates $u_1, \dotsc, u_5$ of iteration \eqref{eqn:Newton} with
$M(u) = F'(u)^{-1} = u^2 + 1$ and step sizes $t_k$ satisfying \eqref{eqn:HlHu},
starting from $u_0 = 2$ for $H = 0.8$. 
Full step Newton diverges for this choice of $u_0$.
For the sake of brevity, the remaining
algorithmic parameters suggested in \cite[section~10.2]{Potschka2016} are used
as explict values in the code and termination and error checks are omitted.

\begin{table}[tb]
  \caption{Output of the Matlab code from Table \ref{alg:example_atan}. The 
  columns display the iteration number $k$, the current step size $t_k$, the
  current iterate $u_k$, the current increment $\delta u_k$, the next trial
  increment $\delta u_k^+$ and the value $H'_k = t_k \norm{g(u_k, t_k)}_2 = t_k
  \norm{\smash[t]{\delta u_k^+ - \delta u_k}}_2$, which must be in
  $[H^\mathrm{l}, H^\mathrm{u}]$ for $t_k < 1$ to be accepted.}
  \label{tab:output}
  \centering
  \begin{verbatim}
  k       t         u        du       dup    Hprime
  0  1.0000   2.0e+00  -5.5e+00   1.7e+01   2.3e+01  decrease t
  0  0.5000   2.0e+00  -5.5e+00   1.0e+00   3.3e+00  decrease t
  0  0.2500   2.0e+00  -5.5e+00  -7.6e-01   1.2e+00  accept t
  1  0.2335   6.2e-01  -7.6e-01  -4.9e-01   6.3e-02  increase t
  1  0.6168   6.2e-01  -7.6e-01  -1.5e-01   3.8e-01  accept t
  2  0.7543   1.5e-01  -1.5e-01  -3.4e-02   8.6e-02  accept t
  3  1.0000   3.4e-02  -3.4e-02   2.7e-05   3.4e-02  accept t
  4  1.0000  -2.7e-05   2.7e-05  -1.3e-14   2.7e-05  accept t
  5  1.0000   1.3e-14  -1.3e-14  -0.0e+00   1.3e-14  accept t
  \end{verbatim}
\end{table}

We display the output in Table \ref{tab:output}. Each line in the output
corresponds to one evaluation of $f(u) = M(u) F(u)$, not counting the initial
evaluation of $f(u_0)$. We observe that for $k = 0$, two bisection steps are
required to reduce $t_0$ to $\frac{1}{4}$. For $k = 1$, one bisection step is
required to increase $t_1$ to above $0.6$. For $k \ge 2$, the predicted step
sizes are accepted without further bisection steps. The iterate $u_5$, which is
known already at the end of iteration 4, is an acceptable solution candidate,
because $\norm{\delta u_5}_2 \approx 1.3 \cdot 10^{-14}$.

The smoothed step size prediction in line 5 is vital for reducing the
number of bracketing steps. The predicted step size often
already satisfies \eqref{eqn:HlHu} in all but a few iterations and thus almost
no extra computational effort in terms of increment evaluations $f(u_k)$, which
typically comprise setting up and solving one linear system, is required for the
globalization procedure in most iterations.

\subsection{Discussion of drawbacks}

Backward step control appears to suffer from two substantial drawbacks: First,
it depends on a problem specific parameter $H > 0$ that needs to be chosen
sufficiently small for convergence, but not too small to cause unecessarily many
iterations and thus inefficiency of the method. Second, it may not converge in
one step on affine linear problems. We argue here that these two seemingly
detrimental properties are actually \emph{necessary} for the class of methods
that converge to the closest solution (in the sense of the Newton flow): Let us
assume we want to solve a one-dimensional nonlinear equation $F(u) = 0$, $F:
\mathbb{R} \to \mathbb{R}$, with a certain \emph{parameter-free} nonlinear
method. Assume it produces a sequence of iterates $(u^k)_{k \in \mathbb{N}}
\subset \mathbb{R}$ that converges to some $u^\ast \neq u^0$. Because $u^0$
cannot be an accumulation point of $(u^k)_{k \in \mathbb{N}}$, we can find the
closest iterate $u^j$ to $u^0$ for some $j \in \mathbb{N}$. Any function that
differs from $F$ only on the open interval $I$ between $u^0$ and $u^j$ will
inevitably lead to the same iterates $(u^k)_{k \in \mathbb{N}}$ if we apply the
parameter-free nonlinear method to it. Thus, we can modify $F$ smoothly on $I$
to introduce zeros of $F$ that are closer to $u^0$ than $u^\ast$. We illustrate
this construction for an affine linear $F$ in Fig.~\ref{fig:drawback}. We
conclude that any method that provably converges to the closest zero must depend
on a parameter such as $H$ to account for problem specific quantities that are
virtually impossible to estimate numerically.

\begin{figure}[tb]
  \sidecaption
  \begin{tikzpicture}
\begin{axis}[xlabel=$u$,ylabel=$F(u)$,height=4cm,width=7cm,
    xmin=-1.2,ymin=-0.2,xmax=1.2,ymax=2.2]
  \addplot[mark=none,color=black,dotted] coordinates {
    (-1.2, 0)
    (1.2, 0)
  };
  \addplot[mark=none,color=black,dashed] coordinates {
    (-1.2, 2.2)
    (1.2, -0.2)
  };
  \addplot[only marks,color=black,mark=x] coordinates {
    (-1,2)
    (1,0)
  } node[pos=0,below] {$F(u^0)$} node[pos=1,above] {$F(u^1)$};
  \addplot[only marks,color=black,mark=+] coordinates {
    ( -7.70893737e-02, 0)
    (  2.70574066e-01, 0)
  };
  \addplot[mark=none,color=black] coordinates {
    ( -5.00000000e-01,   1.50000000e+00)
    ( -4.90000000e-01,   1.49000000e+00)
    ( -4.80000000e-01,   1.47999134e+00)
    ( -4.70000000e-01,   1.46944273e+00)
    ( -4.60000000e-01,   1.45553642e+00)
    ( -4.50000000e-01,   1.43446323e+00)
    ( -4.40000000e-01,   1.40435020e+00)
    ( -4.30000000e-01,   1.36553504e+00)
    ( -4.20000000e-01,   1.31955318e+00)
    ( -4.10000000e-01,   1.26827841e+00)
    ( -4.00000000e-01,   1.21347043e+00)
    ( -3.90000000e-01,   1.15660287e+00)
    ( -3.80000000e-01,   1.09883454e+00)
    ( -3.70000000e-01,   1.04104061e+00)
    ( -3.60000000e-01,   9.83862184e-01)
    ( -3.50000000e-01,   9.27756039e-01)
    ( -3.40000000e-01,   8.73037703e-01)
    ( -3.30000000e-01,   8.19916462e-01)
    ( -3.20000000e-01,   7.68522726e-01)
    ( -3.10000000e-01,   7.18929016e-01)
    ( -3.00000000e-01,   6.71165839e-01)
    ( -2.90000000e-01,   6.25233613e-01)
    ( -2.80000000e-01,   5.81111573e-01)
    ( -2.70000000e-01,   5.38764389e-01)
    ( -2.60000000e-01,   4.98147069e-01)
    ( -2.50000000e-01,   4.59208586e-01)
    ( -2.40000000e-01,   4.21894527e-01)
    ( -2.30000000e-01,   3.86149043e-01)
    ( -2.20000000e-01,   3.51916244e-01)
    ( -2.10000000e-01,   3.19141214e-01)
    ( -2.00000000e-01,   2.87770706e-01)
    ( -1.90000000e-01,   2.57753640e-01)
    ( -1.80000000e-01,   2.29041419e-01)
    ( -1.70000000e-01,   2.01588131e-01)
    ( -1.60000000e-01,   1.75350666e-01)
    ( -1.50000000e-01,   1.50288769e-01)
    ( -1.40000000e-01,   1.26365036e-01)
    ( -1.30000000e-01,   1.03544899e-01)
    ( -1.20000000e-01,   8.17965729e-02)
    ( -1.10000000e-01,   6.10909940e-02)
    ( -1.00000000e-01,   4.14017556e-02)
    ( -9.00000000e-02,   2.27050349e-02)
    ( -8.00000000e-02,   4.97952203e-03)
    ( -7.00000000e-02,  -1.17936478e-02)
    ( -6.00000000e-02,  -2.76309593e-02)
    ( -5.00000000e-02,  -4.25465749e-02)
    ( -4.00000000e-02,  -5.65523876e-02)
    ( -3.00000000e-02,  -6.96580654e-02)
    ( -2.00000000e-02,  -8.18710888e-02)
    ( -1.00000000e-02,  -9.31967799e-02)
    (  0.00000000e+00,  -1.03638324e-01)
    (  1.00000000e-02,  -1.13196780e-01)
    (  2.00000000e-02,  -1.21871089e-01)
    (  3.00000000e-02,  -1.29658065e-01)
    (  4.00000000e-02,  -1.36552388e-01)
    (  5.00000000e-02,  -1.42546575e-01)
    (  6.00000000e-02,  -1.47630959e-01)
    (  7.00000000e-02,  -1.51793648e-01)
    (  8.00000000e-02,  -1.55020478e-01)
    (  9.00000000e-02,  -1.57294965e-01)
    (  1.00000000e-01,  -1.58598244e-01)
    (  1.10000000e-01,  -1.58909006e-01)
    (  1.20000000e-01,  -1.58203427e-01)
    (  1.30000000e-01,  -1.56455101e-01)
    (  1.40000000e-01,  -1.53634964e-01)
    (  1.50000000e-01,  -1.49711231e-01)
    (  1.60000000e-01,  -1.44649334e-01)
    (  1.70000000e-01,  -1.38411869e-01)
    (  1.80000000e-01,  -1.30958581e-01)
    (  1.90000000e-01,  -1.22246360e-01)
    (  2.00000000e-01,  -1.12229294e-01)
    (  2.10000000e-01,  -1.00858786e-01)
    (  2.20000000e-01,  -8.80837555e-02)
    (  2.30000000e-01,  -7.38509574e-02)
    (  2.40000000e-01,  -5.81054729e-02)
    (  2.50000000e-01,  -4.07914143e-02)
    (  2.60000000e-01,  -2.18529306e-02)
    (  2.70000000e-01,  -1.23561111e-03)
    (  2.80000000e-01,   2.11115735e-02)
    (  2.90000000e-01,   4.52336131e-02)
    (  3.00000000e-01,   7.11658385e-02)
    (  3.10000000e-01,   9.89290159e-02)
    (  3.20000000e-01,   1.28522726e-01)
    (  3.30000000e-01,   1.59916462e-01)
    (  3.40000000e-01,   1.93037703e-01)
    (  3.50000000e-01,   2.27756039e-01)
    (  3.60000000e-01,   2.63862184e-01)
    (  3.70000000e-01,   3.01040605e-01)
    (  3.80000000e-01,   3.38834539e-01)
    (  3.90000000e-01,   3.76602871e-01)
    (  4.00000000e-01,   4.13470428e-01)
    (  4.10000000e-01,   4.48278408e-01)
    (  4.20000000e-01,   4.79553180e-01)
    (  4.30000000e-01,   5.05535042e-01)
    (  4.40000000e-01,   5.24350203e-01)
    (  4.50000000e-01,   5.34463227e-01)
    (  4.60000000e-01,   5.35536421e-01)
    (  4.70000000e-01,   5.29442726e-01)
    (  4.80000000e-01,   5.19991337e-01)
    (  4.90000000e-01,   5.10000000e-01)
    (  5.00000000e-01,   5.00000000e-01)
  };
\end{axis}
\end{tikzpicture}
  \caption{Any method that solves affine linear functions with evaluations of
    $F$ in $u^0$ and $u^\ast$ only, i.e., in one step, will generate the iterate
    $u^1 = u^\ast = 1$ for the function $F(u) = 1 - x$ (dashed graph). Thus, it
    will produce the same iterates for the function $F$ modified smoothly on
    $(-1, 1)$ as indicated by the solid curve. Consequently, it will miss the
    zeros closer to $u^0$ indicated with $+$ marks.}
  \label{fig:drawback}
\end{figure}
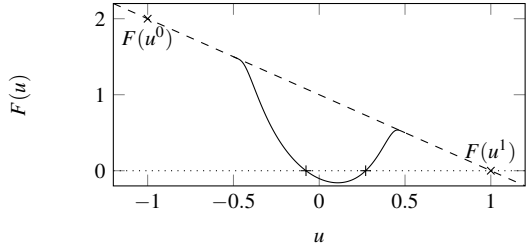

Moreover, existing globalization methods
\cite{Deuflhard1991,Hohmann1994,Deuflhard1998,Deuflhard2006} are not
parameter free either, because they require an initial step size guess $t_0$
that needs to be sufficiently small.

\section{Design of Newton-type methods}
\label{sec:designM}

The convergence analysis of backward step control lends itself immediately to
the design of globally convergent Newton-type methods. The two required steps
are:
\begin{enumerate}
  \item Define $M$ (or directly $f$) respecting the $\kappa$-condition
    A\ref{ass:kappa}.
  \item Use \eqref{eqn:BSC} to determine the step size sequence $(t_k)$.
\end{enumerate}
The second step is generic. For the first step, we give two important examples
in the following two sections.
Both methods find approximations $\delta u_k$ of the Newton increment $\delta
u_k^\mathrm{Newton}$ determined by the (infinite-dimensional) linear system
\begin{equation}
  \label{eqn:HilbertSpaceNewton}
  F'(u_k) \delta u_k^\mathrm{Newton} = -F(u_k).
\end{equation}
We emphasize that the operator $M(u_k)$ is implicitly determined by requiring
$\delta u_k = -M(u_k) F(u_k)$. It is not required to actually compute $M(u_k)$,
as long as we have $\delta u_k$. However, we need to make sure that $f$ is
Lipschitz continuous with respect to $u$.

\subsection{Krylov--Newton methods}
\label{sec:KrylovNewton}

Krylov subspace methods like CG, MINRES, or GMRES
\cite{Hestenes1952,Paige1975,Saad1986} for the iterative solution of linear
systems, originally developed for large but finite-dimensional sparse systems,
can also be stated for infinite-dimensional linear operators. The
convergence theory is more complicated in the infinite-dimensional case (see,
e.g., \cite{Nevanlinna1993,Gasparo2008}). The structure of the linear mapping
$F'(u_k): U \to V$ does usually not immediately admit the application of Krylov
subspace methods, unless a left preconditioner $P_k: V \to U$ is available, such
that $A_k := P_k F'(u_k)$ is an endomorphism on $U$. Then, the $m$-th iterate
$\delta u_k^{m}$ of a Krylov subspace method applied to
\eqref{eqn:HilbertSpaceNewton} is an approximate solution restricted to the (at
most $m$-dimensional) $m$-th Krylov subspace
\[
\mathcal{K}^m(A_k, P_k F(u_k)) = \left\{ q(A_k) P_k F(u_k) \mid
  q \text{ is a polynomial of degree less than } m \right\}.
\]
We now focus on residual minimizing Krylov subspace like GMRES and MINRES, which
are constructed on the basis of the additional optimality requirement
\begin{equation}
  \label{eqn:residual_minimizing}
  \delta u_k^m = \argmin_{\delta u \in \mathcal{K}^m(A_k, P_k F(u_k))}
  \norm{P_k F(u_k) + A_k \delta u_k)}_U.
\end{equation}
As in \cite{Guennel2014}, we investigate the important special case where
$V = U^\ast$. In this setting, we can choose $P_k$ as the Riesz isomorphism and
immediately obtain
\begin{equation}
  \label{eqn:Krylov}
  \begin{aligned}
    \norm{P_k F(u_k) + A_k \delta u_k)}_U &=
    \norm{P_k \left[ F(u_k) + F'(u_k) \delta u_k \right]}_U\\
    &= \norm{F(u_k) + F'(u_k) \delta u_k)}_V.
  \end{aligned}
\end{equation}
In the light of \eqref{eqn:residual_minimizing} and \eqref{eqn:Krylov}, the
$\kappa$-condition A\ref{ass:kappa} is nothing but the classical relative
termination condition used by Krylov subspace methods with given relative
tolerance $\kappa$. In other words, GMRES and MINRES choose from the Krylov
subspace the increment that achieves the smallest left-hand side in the
condition of A\ref{ass:kappa} and thus yield (with Lemma
\ref{lem:linearconvergence}) an asymptotic linear convergence rate bounded above
by $\kappa$ for the nonlinear Krylov--Newton
method. Also for other Krylov space methods like preconditioned CG, the most
commonly used termination criterion is that of the relative residual, even
though the relative residual is not guaranteed to decrease monotonically.
Thus, the backward step convergence theory of section~\ref{sec:convergence}
delivers suitable termination criteria for the inner linear iterations.
In particular, a rather loose relative stopping criterion of, say, $\kappa =
\frac{1}{10}$ delivers asymptotically already one decimal digit of accuracy per
nonlinear iteration.

With this approach, there is one theoretic gap we need to close: The increment
$-f(u_k) = -M(u_k) F(u_k) = \delta u_k^m$ depends on the number of Krylov
subspace iterations $m$, which is determined adaptively. 
As a concatenation of a finite number of Lipschitz continuous operations, the
$m$-th Krylov iterate depends Lipschitz continuously on $u_k$, but changes in
$m$ from one nonlinear $k$-iteration to another can lead to discontinuities in
the operator $M$. However, a small modification of the above approach can make
sure that $M(u)$ and thus $f(u)$ are Lipschitz continuous with respect to $u$,
as required for A\ref{ass:fLipschitz}: Instead of using the final iterate
$\delta u_k^m$, we could use a linear combination $f(u_k) = (1-\alpha_k) \delta
u_k^{m-1} + \alpha_k \delta u_k^m$ of the two last iterates such that instead of
the inequality A\ref{ass:kappa} the equality
\[
\nu(\alpha_k) := \norm{F(u_k) + F'(u_k) \left[ (1-\alpha_k) \delta u_k^{m-1} +
\alpha_k \delta u_k^m \right]}_V = \kappa \norm{F(u_k)}_V
\]
holds for some $\alpha_k \in [0, 1]$. This is always possible by virtue of the
intermediate value theorem applied to the continuous function $\nu(\alpha)$,
which satisfies $\nu(0) > \kappa \norm{F(u_k)}$ and $\nu(1) \le \kappa
\norm{F(u_k)}$ because the Krylov subspace method has terminated in step $m$ but
not yet in step $m-1$. If we now assume that there is an upper bound on the
number of Krylov iterations required to satisfy the relative termination
criterion A\ref{ass:kappa} on the level set of $u^0$, we can establish Lipschitz
continuity of $M$.

Based on our experience in practical computations, however, it is more efficient
to always use $\alpha_k = 1$ and robustify the bisection procedure for the
approximate solution to \eqref{eqn:BSC} against discontinuities of $g$ by
relaxing the lower bound $H^\mathrm{l}$ in \eqref{eqn:HlHu} closer to zero. This
might give rise to smaller than necessary step sizes $t_k$, which has not been
observed to be problematic in practical computations, but usually delivers
faster local residual contraction once $t_k = 1$.

We report numerical results of a Krylov--Newton method for the Carrier equation
in section~\ref{sec:carrier}.

\subsection{Approximation by discretization}
\label{sec:M_by_discretization}

Following the multilevel Newton approach of \cite{Hohmann1994},
we can also compute an approximate solution $\delta u_k$ by first discretizing
\eqref{eqn:HilbertSpaceNewton} and then (approximately) solving the discretized
system. The $\kappa$-condition A\ref{ass:kappa} yields a computable criterion
for checking if the approximation is accurate enough to ensure convergence in
$U$. If not, we need to improve the discretization (and possibly the accuracy of
the approximate solution to the resulting finite-dimensional linear system).

In this conceptually simple approach, challenges can arise in the evaluation of
the $V$-norms in A\ref{ass:kappa}. We address these
issues for the case of an elliptic partial differential equation in
section~\ref{sec:minsurf}. Moreover, the evaluation of the $V$-norm in
A\ref{ass:kappa} can provide a means to adaptively discretize
\eqref{eqn:HilbertSpaceNewton} and in turn also the original nonlinear problem.

As in section~\ref{sec:KrylovNewton}, the procedure for approximating a solution to
\eqref{eqn:HilbertSpaceNewton} usually involves some discrete decisions, for
instance the marking and refinement of certain discretization cells as $k$
increases. Thus, the so constructed $f(u)$ is not Lipschitz continuous. 
A smoothed formulation akin to the interpolation construction in
section~\ref{sec:KrylovNewton} exceeds the scope of this article and shall be
investigated in future work.

\section{Numerical examples from nonlinear elliptic boundary value problems}
\label{sec:ellipticPDE}

In this section, we illustrate the general paradigms presented in
section~\ref{sec:M_by_discretization} for the class of elliptic boundary value problems
on a bounded domain $\Omega \subset \R^n$ with continuously differentiable
boundary $\partial \Omega$ that can be cast as nonlinear root-finding problems
\eqref{eqn:FOfxIsZero} with the Sobolev spaces $U = H^1_0(\Omega)$ and $V =
H^{-1}(\Omega)$.   

Based on the Poincar\'e inequality (see, e.g., \cite{Evans2010}), we can use the
inner product 
\[
\inprod{u,v}_U = \int_{\Omega} \nabla u \cdot \nabla v
\]
for the Hilbert space $U$.  For the Hilbert space $V$, we can then compute norms
via the Riesz representation theorem \cite[\S III.3]{Yosida1978}: For every $v
\in V$, there exists a uniquely determined $r_v \in U$ such that 
\begin{equation}
  \label{eqn:RieszRepresentation}
  \inprod{u, r_v}_U = \int_{\Omega} v u \text{ for all } u \in U
  \quad \text{and} \quad
  \norm{v}_V = \norm{r_v}_U.
\end{equation}

We investigate the numerical performance of backward step control on two
nonlinear elliptic boundary value problems on bounded domains $\Omega$ with
continuously differentiable boundary.

All algorithmic parameters of backward step control are chosen as in
\cite{Potschka2016} unless otherwise stated.

\subsection{Preconditioned Krylov subspace methods}
\label{sec:PGMRES}

We first illustrate the general Krylov subspace method approach presented in
section~\ref{sec:KrylovNewton}. 
As our focus in the case of Krylov--Newton methods here lies on a concise
algorithmic statement rather than ultimate computational speed, we use
Chebfun \cite{Battles2004,Driscoll2014} as an algorithmic tool for the numerical
results in section~\ref{sec:carrier}, because it allows to compute numerically
with functions (represented as adaptively truncated Chebyshev expansions)
instead of numbers \cite{Trefethen2007} and supports the automatic computation
of Fr\'echet derivatives by the use of automatic differentiation in function
space \cite{Birkisson2012}.

Because the linear operators in Chebfun are implemented in strong form, we also
use the strong form of the inner product in $U$
\[
\inprod{u,v}_U = -\int_{-1}^{1} u \Lap v = -\int_{-1}^{1} v \Lap u, 
\]
(which requires $u$ or $v$ to have square integrable second derivatives).  It
follows from \eqref{eqn:RieszRepresentation} that 
\[
-\int_{-1}^{1} u \Lap r_v = \inprod{u, r_v}_U = \int_{-1}^{1} v u 
\quad \text{for all } u \in U,
\]
and, thus, $r_v = -\Lap^{-1} v$ and
\[
\norm{v}_V^2 = \norm{r_v}_U^2 = \norm{-\Lap^{-1} v}_U^2.
\]
In Chebfun, $\Lap^{-1}$ can be evaluated efficiently using ultraspherical
spectral collocation \cite{Olver2013}. 
Based on these prerequisites, we modified Chebfun's builtin GMRES to use the
inner product and norm of $U$ (instead of $L^2(\Omega)$) in combination with
$\Lap^{-1}$ as a preconditioner, which yields the correct residual norm
$\norm{\Lap^{-1}v}_U = \norm{v}_V$. 
We note that we could have used MINRES or even CG because the resulting
left-preconditioned linear system is self-adjoint and positive definite in the
case at hand \cite{Guennel2014}. This does not affect our results dramatically,
because GMRES produces the same iterates as MINRES for self-adjoint systems.
However, the current version of Chebfun does not ship a MINRES implementation.
The numerical results for the Carrier equation in section~\ref{sec:carrier}
indicate that using the Riesz isomorphism as a preconditioner works
satisfactorily in our example.

\subsubsection{Application to the Carrier equation}
\label{sec:carrier}

\begin{figure}[bt]
  \begin{center}
    \input{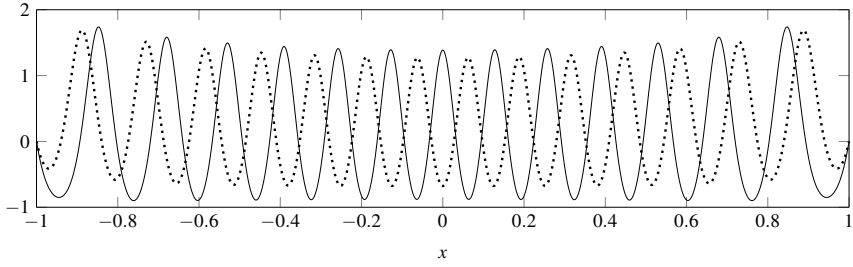}
  \end{center}
  \caption{The local solutions to the Carrier equation with $\varepsilon =
  10^{-3}$ obtained by a GMRES--Newton method with backward step control for
  $\kappa = 10^{-2}$ and varying values of $H_\mathrm{rel}$ (0.1: solid, 0.05
  and 0.01: dotted).}
  \label{fig:carrier_solution}
\end{figure}

For $\varepsilon > 0$, we want to determine a real-valued function $u(x)$ on
$x \in [-1, 1]$ that solves the nonlinear second order boundary value problem
\begin{equation}
  \label{eqn:carrier}
  \varepsilon \Lap u + 2 (1-x^2) u + u^2 = 1, \quad u(\pm 1) = 0,
\end{equation}
which---according to \cite[\S 9.7]{Bender1999}---is due to Carrier. For small
$\varepsilon$, it becomes challenging to solve \eqref{eqn:carrier} because of
the existence of many local solutions (compare Fig.~\ref{fig:carrier_solution}).

We apply the GMRES--Newton method described in section~\ref{sec:PGMRES} to
\eqref{eqn:carrier} and illustrate the
convergence of the method in Fig.~\ref{fig:carrier_convergence} for $\varepsilon
= 10^{-3}$, $\kappa = 10^{-2}$, and varying values of $H = H_\mathrm{rel}
\norm{\delta u_0}_U$, where we choose $H_\mathrm{rel} \in \{ 0.5, 0.1, 0.05,
0.01 \}$. The initial guess is $u_0 = 0$ and the termination criterion is
$\norm{F(x_k)}_V \le 10^{-11}$. The relative GMRES termination tolerance
$\kappa$ was chosen rather large but at the same time small enough to ensure
sufficiently fast local convergence with two decimal digits per iteration.
We first observe that no convergence can be obtained for
$H_\mathrm{rel} = 0.5$, even though $\norm{F(u_{41})}_V$ drops below
$2.6 \cdot 10^{-5}$ and $\norm{\delta u_7}_U \approx 0.28$. From these numbers
we can estimate the nonlinearity of the problem in terms of $\omega$ and its
well-posedness in terms of $r$ based on Lemma \ref{lem:Fdecrease}, which yields
\[
\omega \ge 2 (1 - \kappa) / (t_7 \norm{f(u_7)}_U) \approx 7.1
\quad \text{and} \quad
\omega r \ge 2 (1 - \kappa) / (t_{41} \norm{F(u_{41})}_V) \approx 7.6 \cdot
10^{4}
\]
and shows that the problem is highly nonlinear.

For the remaining choices of $H_\mathrm{rel}$ we obtain convergence, albeit a
different local solution is found for $H_\mathrm{rel} = 0.1$ than for the
others (compare Fig.~\ref{fig:carrier_solution}), which nicely confirms the
theory of Theorem \ref{thm:convergenceBSC}. As guaranteed by Lemma
\ref{lem:localFullSteps}, full steps $t_k = 1$ are taken in the vicinity of a
solution and we can clearly observe the asymptotic linear convergence rate of
$\kappa = 10^{-2}$ for the residual norm predicted by Lemma
\ref{lem:linearconvergence}. The final increment norms seem rather
large because we do not compute $\delta u_k$ if $\norm{F(u_k)}_V$ is already
below $10^{-11}$. Thus, the last increment norm lags behind by one iteration and
would be much smaller if we computed it again for the final iterate.

In Fig.~\ref{fig:carrier_gmres}, we see that the number of GMRES iterations
needed in each nonlinear iteration stays moderately small. In each GMRES
iteration, one operator-vector-multiplication must be carried out, which we 
compute via Chebfun as a directional derivative of $F$. In total, 1255
($H_\mathrm{rel} = 0.1$), 1455 ($H_\mathrm{rel} = 0.05$), and 2471
($H_\mathrm{rel} = 0.01$) directional derivatives of $F$ are required to solve
the problem. 

In addition, we can observe from Fig.~\ref{fig:carrier_gmres} that the predicted
step size $t_k$ needs to be corrected only in few iterations $k$ by the
backward step control bisection procedure outlined in section~\ref{sec:realization}.
In the first iteration, four ($H_\mathrm{rel} = 0.1, 0.05$) and five
($H_\mathrm{rel} = 0.01$) bisection steps are required to reduce the initial
step size guess $t_0 = 1$. In all other iterations marked by $\bullet$ in
Fig.~\ref{fig:carrier_gmres}, only one additional bisection step is necessary,
except in iteration $k=8$ for $H_\mathrm{rel} = 0.1$ and $k=21$ for
$H_\mathrm{rel} = 0.01$, where two bisection steps need to be taken. This backs
up our claim that the computational overhead of backward step control is small.

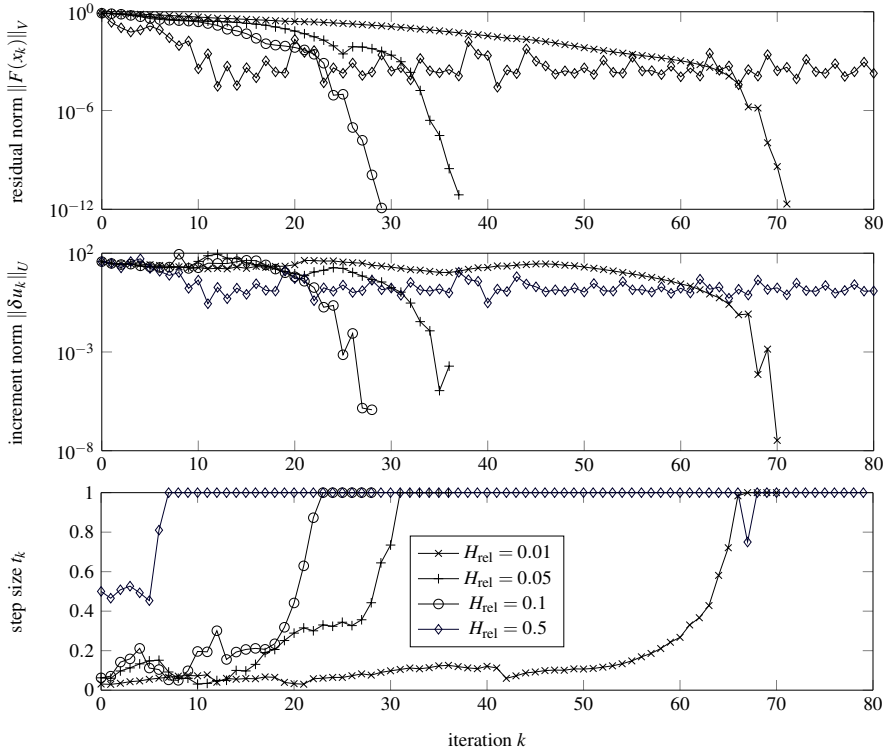
\begin{figure}[tb]
  \begin{center}
%
%
\definecolor{mycolor1}{rgb}{0.00000,0.00000,0.17241}%
\begin{tikzpicture}[scale=0.9]

\begin{axis}[%
width=0.95\textwidth,
height=0.15\textheight,
at={(0\textwidth,0\textheight)},
scale only axis,
separate axis lines,
every outer x axis line/.append style={black},
every x tick label/.append style={font=\color{black}},
xmin=0,
xmax=80,
every outer y axis line/.append style={black},
every y tick label/.append style={font=\color{black}},
ymode=log,
ymin=1e-12,
ymax=1,
yminorticks=true,
ylabel={residual norm $\norm{F(x_k)}_V$},
axis background/.style={fill=white}
]
\addplot [color=black,solid,mark=x,mark options={solid},forget plot]
  table[row sep=crcr]{%
0	0.816496580927726\\
1	0.790029899109026\\
2	0.766291434930829\\
3	0.738886570758214\\
4	0.706448641034833\\
5	0.672680665484149\\
6	0.634909088511211\\
7	0.595517913946676\\
8	0.55589399669793\\
9	0.515657320507544\\
10	0.476701694398399\\
11	0.441483490371933\\
12	0.406564432646215\\
13	0.390144986208951\\
14	0.366583980446912\\
15	0.34598756327736\\
16	0.325392761841446\\
17	0.30623822485458\\
18	0.285292612017863\\
19	0.266383830282909\\
20	0.255705167147134\\
21	0.247132898777281\\
22	0.23935033612242\\
23	0.224310615868374\\
24	0.20953689433802\\
25	0.19465566365976\\
26	0.181070493002611\\
27	0.166705891446427\\
28	0.151624711599626\\
29	0.139376653020254\\
30	0.12600228756203\\
31	0.112900845837282\\
32	0.100395299680629\\
33	0.0885485558393549\\
34	0.0784675066257065\\
35	0.0690962802399877\\
36	0.0601764590965218\\
37	0.0522056168847492\\
38	0.04549362035724\\
39	0.0397410481317883\\
40	0.034634666595682\\
41	0.0294597848503061\\
42	0.0249668753799962\\
43	0.0231588010298584\\
44	0.0209140086404864\\
45	0.0181778999069939\\
46	0.0153478289327861\\
47	0.0125415003254889\\
48	0.0100949181951477\\
49	0.00820320259153092\\
50	0.0065186609891555\\
51	0.00530399270202473\\
52	0.00437364814026496\\
53	0.00363399061520742\\
54	0.00305501058054998\\
55	0.00258264315585463\\
56	0.00218644646181106\\
57	0.00184458868653958\\
58	0.00155365270455214\\
59	0.00129220607329828\\
60	0.0010512705506538\\
61	0.000826659383820676\\
62	0.00061248874043716\\
63	0.000422534432127693\\
64	0.000260597652419821\\
65	0.000124109670535846\\
66	4.00479557317963e-05\\
67	1.66244719349787e-06\\
68	1.41789594820098e-06\\
69	1.09210549789929e-08\\
70	3.91739108606084e-10\\
71	2.08658726354032e-12\\
};
\addplot [color=black,solid,mark=+,mark options={solid},forget plot]
  table[row sep=crcr]{%
0	0.816496580927726\\
1	0.761658952573407\\
2	0.713850212791108\\
3	0.641426779743992\\
4	0.565506565435472\\
5	0.487355166833662\\
6	0.412368170891491\\
7	0.347336374067904\\
8	0.314642236581349\\
9	0.295367868626837\\
10	0.277262657800289\\
11	0.268372414071649\\
12	0.257790048701072\\
13	0.241139526405461\\
14	0.227528629111677\\
15	0.199093711007248\\
16	0.177681563941134\\
17	0.152133914423233\\
18	0.119286570173193\\
19	0.0923812093263947\\
20	0.0672183102663504\\
21	0.0462280471549922\\
22	0.030335786141461\\
23	0.0186983201344002\\
24	0.00846764323392152\\
25	0.00275547976501476\\
26	0.007476178597726\\
27	0.00721610278281882\\
28	0.00572500000268585\\
29	0.00393838396129286\\
30	0.00229498449224849\\
31	0.000920972742201224\\
32	0.000207846577751118\\
33	1.64777692012789e-05\\
34	2.52969398170233e-07\\
35	3.0237027482978e-08\\
36	2.94893642854542e-10\\
37	7.54761635881926e-12\\
};
\addplot [color=black,solid,mark=o,mark options={solid},forget plot]
  table[row sep=crcr]{%
0	0.816496580927726\\
1	0.761658952573407\\
2	0.705202126541156\\
3	0.598689838885752\\
4	0.497497828643865\\
5	0.38574334493679\\
6	0.341227101372542\\
7	0.304764669262662\\
8	0.28865063700174\\
9	0.270061354644358\\
10	0.242991026048848\\
11	0.192649169669352\\
12	0.147755303522577\\
13	0.0861210357554966\\
14	0.0686615658377321\\
15	0.0463121872214602\\
16	0.0243460855084114\\
17	0.0114429982490634\\
18	0.00939008081430182\\
19	0.00821774733582876\\
20	0.00674876799218295\\
21	0.00480596997030786\\
22	0.00260664259218833\\
23	0.000750612228000923\\
24	8.488866281442e-06\\
25	9.71085512713389e-06\\
26	9.28671032531154e-08\\
27	1.55183395080086e-08\\
28	1.22634519105052e-10\\
29	1.1973504396654e-12\\
};
\addplot [color=black,solid,mark=diamond,mark options={solid},forget plot]
  table[row sep=crcr]{%
0	0.816496580927726\\
1	0.225149297415132\\
2	0.105805997513201\\
3	0.0585016360025852\\
4	0.0782092793698187\\
5	0.13514589355092\\
6	0.0801844387974075\\
7	0.0263103147055343\\
8	0.00876238614601008\\
9	0.0166914790247353\\
10	0.000341538296576455\\
11	0.00282219464134327\\
12	2.93448369952672e-05\\
13	0.000497573947119082\\
14	3.30524157553982e-05\\
15	0.000400422624783008\\
16	9.76393974158192e-05\\
17	0.00104391438265612\\
18	0.000222646283094222\\
19	0.000192402613027569\\
20	0.0209285114723644\\
21	0.00335530025908698\\
22	0.00442915375317871\\
23	4.83123983521153e-05\\
24	0.000393833796272231\\
25	0.000183015756256064\\
26	0.000767710207708295\\
27	0.000134756429740426\\
28	0.000215318984114684\\
29	0.00235904480846628\\
30	0.000248787628838208\\
31	0.000400637346223648\\
32	7.23883723678109e-05\\
33	0.00145950487911155\\
34	0.000260372118210984\\
35	0.000194463720996965\\
36	0.000304442160271582\\
37	0.000124644050777556\\
38	0.0151961315184583\\
39	0.00263031887286565\\
40	0.00219164732194238\\
41	2.5976761081687e-05\\
42	0.000287837514126991\\
43	0.000173536041006381\\
44	0.00550858160655769\\
45	0.000998447850133191\\
46	0.000509470091736138\\
47	0.000163703491480122\\
48	0.00025678875952082\\
49	0.00018048645888191\\
50	0.000391144783133695\\
51	0.000154275387469895\\
52	0.00112925608380718\\
53	0.000210957311330048\\
54	0.000185060831744654\\
55	0.00172630296818113\\
56	0.000251408726176748\\
57	0.000264277903890273\\
58	0.000151761934654547\\
59	0.00048758910381955\\
60	0.000112166249970398\\
61	0.000372303749693597\\
62	0.000123449325356246\\
63	0.00300980542630632\\
64	0.000276561469665887\\
65	0.000508199474639872\\
66	3.68652735572562e-05\\
67	0.000304971501342191\\
68	0.000122685039650272\\
69	0.00248648021969636\\
70	0.000259602696024292\\
71	0.000422268399766936\\
72	8.32554620920244e-05\\
73	0.0012746935650094\\
74	0.000228025614313208\\
75	0.000182844562579884\\
76	0.000701404348051139\\
77	0.000117680445088894\\
78	0.000224914696576807\\
79	0.000867807524493437\\
80	0.000182030778298373\\
};
\end{axis}
\end{tikzpicture}%
%
%
\definecolor{mycolor1}{rgb}{0.00000,0.00000,0.17241}%
\begin{tikzpicture}[scale=0.9]

\begin{axis}[%
width=0.95\textwidth,
height=0.15\textheight,
at={(0\textwidth,0\textheight)},
scale only axis,
separate axis lines,
every outer x axis line/.append style={black},
every x tick label/.append style={font=\color{black}},
xmin=0,
xmax=80,
every outer y axis line/.append style={black},
every y tick label/.append style={font=\color{black}},
ymode=log,
ymin=1e-08,
ymax=100,
yminorticks=true,
ylabel={increment norm $\norm{\delta u_k}_U$},
axis background/.style={fill=white}
]
\addplot [color=black,solid,mark=x,mark options={solid},forget plot]
  table[row sep=crcr]{%
0	37.430435786285\\
1	31.4787423079066\\
2	30.0030219236812\\
3	28.9972943027639\\
4	27.9670237690716\\
5	26.2423689682985\\
6	23.9984321088538\\
7	21.5015895546088\\
8	20.4477719323353\\
9	17.8665896137666\\
10	18.8441989993064\\
11	16.2675015954933\\
12	17.2246427657141\\
13	16.1219543916851\\
14	20.2328513806461\\
15	17.163355836345\\
16	21.4629945660428\\
17	19.213213899539\\
18	21.510850680739\\
19	22.0668403505425\\
20	26.103847293689\\
21	42.9753436214261\\
22	42.4624655989726\\
23	40.0569633197268\\
24	40.6158827810234\\
25	35.7307802800291\\
26	33.1578401842591\\
27	31.4187057923083\\
28	24.5122280475483\\
29	22.471723581391\\
30	19.79796557186\\
31	17.1369365777225\\
32	14.6666600464122\\
33	13.8545154116231\\
34	12.03846785883\\
35	10.7096295948776\\
36	10.5368430412412\\
37	12.3783646387821\\
38	15.2952024526389\\
39	18.2802754516665\\
40	18.9304051976503\\
41	22.6625822201856\\
42	22.2339449596415\\
43	24.9111669330091\\
44	26.6061186336948\\
45	28.6867725991092\\
46	28.6616145339346\\
47	27.3549659113163\\
48	24.6538276981583\\
49	23.0233221999104\\
50	19.9805245201778\\
51	17.2997229302239\\
52	14.9343713571493\\
53	12.6085119472934\\
54	10.5940635793902\\
55	8.84838781957047\\
56	7.41434946548517\\
57	5.98136876918343\\
58	4.79440439740139\\
59	3.7970111348149\\
60	2.78274768894217\\
61	2.14251199956295\\
62	1.41555047763194\\
63	0.863996050598514\\
64	0.552072359591071\\
65	0.260004427763641\\
66	0.0765670237542317\\
67	0.0831259604471272\\
68	7.30555733791556e-05\\
69	0.00139173923746009\\
70	3.33974524649766e-08\\
71	0\\
};
\addplot [color=black,solid,mark=+,mark options={solid},forget plot]
  table[row sep=crcr]{%
0	37.430435786285\\
1	30.3286144389862\\
2	26.9817708720494\\
3	28.3193370454918\\
4	22.2178000221338\\
5	20.4114517387064\\
6	20.4817587627352\\
7	19.0537187370784\\
8	19.8722334563073\\
9	19.7468932904568\\
10	38.5349426201494\\
11	74.4772817550999\\
12	89.3057399747097\\
13	55.6161726846725\\
14	53.3010556295855\\
15	32.3194243078049\\
16	25.6161613281099\\
17	22.4276154572153\\
18	16.050204091583\\
19	11.9486542411119\\
20	9.15830379720423\\
21	7.34380771402011\\
22	11.2174105760422\\
23	13.6485182847562\\
24	18.5012836975328\\
25	16.9570198929751\\
26	10.6235374275039\\
27	6.97707492471778\\
28	4.7006706205853\\
29	3.48395257126185\\
30	1.80997011762176\\
31	1.06785811373832\\
32	0.299427697151989\\
33	0.034326204947871\\
34	0.0117927805341467\\
35	1.10445571350834e-05\\
36	0.000191466006291712\\
37	0\\
};
\addplot [color=black,solid,mark=o,mark options={solid},forget plot]
  table[row sep=crcr]{%
0	37.430435786285\\
1	30.3286144389862\\
2	26.596850340412\\
3	27.9349565576616\\
4	22.9148569773756\\
5	20.4334971490533\\
6	18.9182020799561\\
7	19.8816795226458\\
8	88.5356385783666\\
9	17.0116159002262\\
10	18.3857017631849\\
11	31.2541320522077\\
12	30.4608103404345\\
13	29.9556069960515\\
14	37.2738732568655\\
15	43.7175555567537\\
16	41.8130810247923\\
17	25.6315288374931\\
18	15.698289173392\\
19	10.276561471802\\
20	6.43456016891692\\
21	3.80298832530335\\
22	1.8234519157997\\
23	0.184367658388349\\
24	0.22590094407561\\
25	0.000719878547595099\\
26	0.00869496802958201\\
27	1.44260088420307e-06\\
28	1.18539647704533e-06\\
29	0\\
};
\addplot [color=mycolor1,solid,mark=diamond,mark options={solid},forget plot]
  table[row sep=crcr]{%
0	37.430435786285\\
1	30.8921647522089\\
2	18.1467619236214\\
3	37.9880823654385\\
4	52.0292840821355\\
5	16.9492148123846\\
6	12.1667854913835\\
7	7.55894465866186\\
8	10.511700035908\\
9	1.66164526404373\\
10	4.3417298435419\\
11	0.277015677209668\\
12	1.86830900704346\\
13	0.497520068746867\\
14	1.70111874912313\\
15	0.826031643029343\\
16	2.84407739504821\\
17	1.31642665918641\\
18	1.20459785979306\\
19	12.7016035923291\\
20	5.16820714898625\\
21	5.20467546723727\\
22	0.395458747277005\\
23	1.77525123334551\\
24	1.18805282961328\\
25	2.43383111376212\\
26	1.02802420479967\\
27	1.27697182480706\\
28	4.26537508283402\\
29	1.41981325813368\\
30	1.71693695139515\\
31	0.738592819783858\\
32	3.31646382544831\\
33	1.37873642243752\\
34	1.22376683194714\\
35	1.52808105318734\\
36	0.979399793470928\\
37	10.8186664241465\\
38	4.5790204020179\\
39	3.80275426933817\\
40	0.303192125266732\\
41	1.5242084660365\\
42	1.15111467965893\\
43	6.51568564726748\\
44	2.75652537475319\\
45	1.98661833364964\\
46	1.11643687525988\\
47	1.41105295822788\\
48	1.17946821836493\\
49	1.73176975510377\\
50	1.08512466378966\\
51	2.92895849949885\\
52	1.25145283996827\\
53	1.19683841487631\\
54	3.64877734057053\\
55	1.41251129610948\\
56	1.40405100454454\\
57	1.08472769881016\\
58	1.94232966402841\\
59	0.93425857458221\\
60	1.68319948081077\\
61	0.968114248535164\\
62	4.82235874500829\\
63	1.50913702378205\\
64	1.92579338323734\\
65	0.525576407152297\\
66	1.48990792839888\\
67	0.762776241446144\\
68	4.38374626891809\\
69	1.45189511064613\\
70	1.76043424103011\\
71	0.787167809862942\\
72	3.07292789796803\\
73	1.27589938242668\\
74	1.18178551454467\\
75	2.32705803072686\\
76	0.961109220038569\\
77	1.3036713017023\\
78	2.58483395351237\\
79	1.19194823398641\\
80	1.23421177497891\\
};
\end{axis}
\end{tikzpicture}%
%
\definecolor{mycolor1}{rgb}{0.00000,0.00000,0.17241}%
\begin{tikzpicture}[scale=0.9]

\begin{axis}[%
width=0.95\textwidth,
height=0.15\textheight,
at={(0\textwidth,0\textheight)},
scale only axis,
separate axis lines,
every outer x axis line/.append style={black},
every x tick label/.append style={font=\color{black}},
xmin=0,
xmax=80,
xlabel={iteration $k$},
every outer y axis line/.append style={black},
every y tick label/.append style={font=\color{black}},
ymin=0,
ymax=1,
ylabel={step size $t_k$},
legend style={at={(0.5,0.5)},anchor=center},
axis background/.style={fill=white}
]
\addplot [color=black,solid,mark=x,mark options={solid}]
  table[row sep=crcr]{%
0	0.03125\\
1	0.029405265334551\\
2	0.0350169692043168\\
3	0.0429418763111966\\
4	0.0468063334453352\\
5	0.0549601921822576\\
6	0.0607878383459011\\
7	0.0652925767239458\\
8	0.0710229368554477\\
9	0.0742977747265594\\
10	0.0726482645860256\\
11	0.0779396369128086\\
12	0.0400622435643463\\
13	0.0597306106249894\\
14	0.0553426739525477\\
15	0.0588022339927368\\
16	0.0577863419028723\\
17	0.067194743444924\\
18	0.0648684757212297\\
19	0.03951878568923\\
20	0.0329520495267775\\
21	0.0301179688450966\\
22	0.057773437539969\\
23	0.0606717161535636\\
24	0.0645921056972332\\
25	0.0645381361080038\\
26	0.0730796847909692\\
27	0.082521944766837\\
28	0.0762426279254945\\
29	0.0900823583762144\\
30	0.0979330905452812\\
31	0.104887275872016\\
32	0.112355536269452\\
33	0.108185760814611\\
34	0.113926324230902\\
35	0.123210289818608\\
36	0.125656880317496\\
37	0.119260365071186\\
38	0.112786739966433\\
39	0.108849379394394\\
40	0.120776236964413\\
41	0.112526477889296\\
42	0.0599068819758072\\
43	0.0725159666089527\\
44	0.0871515308789389\\
45	0.092147215019713\\
46	0.0991082989628558\\
47	0.101715535979133\\
48	0.100802582362082\\
49	0.107988813504275\\
50	0.106953267600522\\
51	0.109997070268303\\
52	0.116499134972432\\
53	0.123216323243081\\
54	0.133793180223917\\
55	0.148139160886382\\
56	0.168824542761312\\
57	0.185967240761688\\
58	0.210715510014618\\
59	0.242621086656612\\
60	0.267355528777775\\
61	0.330112497159648\\
62	0.366687018252007\\
63	0.428327179226344\\
64	0.581069052679001\\
65	0.72102248428647\\
66	0.983937225263567\\
67	1\\
68	1\\
69	1\\
70	1\\
};
\addlegendentry{$H_\mathrm{rel} = 0.01$};

\addplot [color=black,solid,mark=+,mark options={solid}]
  table[row sep=crcr]{%
0	0.0625\\
1	0.0604961336720751\\
2	0.097059998141582\\
3	0.112857203353862\\
4	0.133166542728677\\
5	0.148315815152934\\
6	0.151579526972115\\
7	0.0920454421132698\\
8	0.0602917061013229\\
9	0.0602917061013229\\
10	0.0310012758317014\\
11	0.0343495519921149\\
12	0.0490477042504068\\
13	0.0500105234274495\\
14	0.100021046854899\\
15	0.0974207081032344\\
16	0.130687758835698\\
17	0.190511894281437\\
18	0.205927661487928\\
19	0.251322862138225\\
20	0.288830353367963\\
21	0.315229083838555\\
22	0.300616869522124\\
23	0.330114109981821\\
24	0.323457278034533\\
25	0.343954597151469\\
26	0.326948460224505\\
27	0.356470334439515\\
28	0.442761631784082\\
29	0.64449997109039\\
30	0.734679324444884\\
31	1\\
32	1\\
33	1\\
34	1\\
35	1\\
36	1\\
};
\addlegendentry{$H_\mathrm{rel} = 0.05$};

\addplot [color=black,solid,mark=o,mark options={solid}]
  table[row sep=crcr]{%
0	0.0625\\
1	0.0709922673441502\\
2	0.1419845346883\\
3	0.158485147388878\\
4	0.211665839840258\\
5	0.111996736248291\\
6	0.104241703895652\\
7	0.0521208519478261\\
8	0.049837372666443\\
9	0.0980319152177586\\
10	0.196063830435517\\
11	0.194065629801214\\
12	0.300964531963525\\
13	0.15542498260279\\
14	0.192925102047323\\
15	0.206128956655181\\
16	0.211600418242349\\
17	0.206787262594986\\
18	0.234887032012483\\
19	0.318381631654115\\
20	0.44157645246533\\
21	0.629438502374635\\
22	0.873053577987544\\
23	1\\
24	1\\
25	1\\
26	1\\
27	1\\
28	1\\
};
\addlegendentry{$H_\mathrm{rel} = 0.1$};

\addplot [color=mycolor1,solid,mark=diamond,mark options={solid}]
  table[row sep=crcr]{%
0	0.5\\
1	0.465499759430172\\
2	0.508318595780802\\
3	0.52661490711337\\
4	0.492850559832814\\
5	0.453239850488481\\
6	0.810252081316412\\
7	1\\
8	1\\
9	1\\
10	1\\
11	1\\
12	1\\
13	1\\
14	1\\
15	1\\
16	1\\
17	1\\
18	1\\
19	1\\
20	1\\
21	1\\
22	1\\
23	1\\
24	1\\
25	1\\
26	1\\
27	1\\
28	1\\
29	1\\
30	1\\
31	1\\
32	1\\
33	1\\
34	1\\
35	1\\
36	1\\
37	1\\
38	1\\
39	1\\
40	1\\
41	1\\
42	1\\
43	1\\
44	1\\
45	1\\
46	1\\
47	1\\
48	1\\
49	1\\
50	1\\
51	1\\
52	1\\
53	1\\
54	1\\
55	1\\
56	1\\
57	1\\
58	1\\
59	1\\
60	1\\
61	1\\
62	1\\
63	1\\
64	1\\
65	1\\
66	1\\
67	0.75\\
68	1\\
69	1\\
70	1\\
71	1\\
72	1\\
73	1\\
74	1\\
75	1\\
76	1\\
77	1\\
78	1\\
79	1\\
};
\addlegendentry{$H_\mathrm{rel} = 0.5$};

\end{axis}
\end{tikzpicture}%
  \end{center}
  \caption{The residual and increment norms and the step size sequence for a
  GMRES--Newton method with backward step control applied to the Carrier
  equation with $\varepsilon = 10^{-3}$ for $\kappa = 10^{-2}$ and varying
  values of $H_\mathrm{rel}$.}
  \label{fig:carrier_convergence}
\end{figure}

\begin{figure}[tb]
  \sidecaption
%
\begin{tikzpicture}[scale=0.9]

\begin{axis}[%
width=0.55\textwidth,
height=0.17\textheight,
at={(0\textwidth,0\textheight)},
scale only axis,
separate axis lines,
every outer x axis line/.append style={black},
every x tick label/.append style={font=\color{black}},
xmin=0,
xmax=70,
xlabel={iteration $k$},
every outer y axis line/.append style={black},
every y tick label/.append style={font=\color{black}},
ymin=10,
ymax=55,
ylabel={GMRES iterations},
legend style={at={(0.73,0.24)},anchor=center},
axis background/.style={fill=white}
]
\addplot [color=black,solid,mark=x,mark options={solid}]
  table[row sep=crcr]{%
0	16\\
1	21.8333333333333\\
2	20\\
3	20\\
4	21\\
5	21\\
6	21\\
7	21\\
8	22\\
9	22\\
10	23\\
11	24\\
12	26\\
13	27.5\\
14	28\\
15	28\\
16	30\\
17	30\\
18	31\\
19	33\\
20	35.5\\
21	35.6666666666667\\
22	36\\
23	36\\
24	36\\
25	34\\
26	33\\
27	34\\
28	32\\
29	32\\
30	32\\
31	32\\
32	32\\
33	33\\
34	33\\
35	33\\
36	33\\
37	33\\
38	33\\
39	33\\
40	32\\
41	33\\
42	31\\
43	33\\
44	33\\
45	33\\
46	33\\
47	33\\
48	33\\
49	34\\
50	34\\
51	34\\
52	34\\
53	34\\
54	34\\
55	34\\
56	33\\
57	34\\
58	34\\
59	34\\
60	32\\
61	34\\
62	34\\
63	34\\
64	36\\
65	36\\
66	34\\
67	48\\
68	14\\
69	53\\
70	15\\
};
\addlegendentry{$H_\mathrm{rel} = 0.01$};
\addplot [color=black,only marks,mark=*,mark options={solid},forget plot]
  table[row sep=crcr]{%
1	21.8333333333333\\
13	27.5\\
20	35.5\\
21	35.6666666666667\\
43	33\\
};
\addplot [color=black,solid,mark=+,mark options={solid}]
  table[row sep=crcr]{%
0	16\\
1	22.2\\
2	21\\
3	23\\
4	22\\
5	24\\
6	27\\
7	29\\
8	34\\
9	35\\
10	39.5\\
11	38.5\\
12	37\\
13	37\\
14	34\\
15	31\\
16	29\\
17	29\\
18	28\\
19	28\\
20	29\\
21	30\\
22	35\\
23	36\\
24	40\\
25	41\\
26	37\\
27	33\\
28	31\\
29	33\\
30	35\\
31	33\\
32	37\\
33	41\\
34	52\\
35	30\\
36	52\\
};
\addlegendentry{$H_\mathrm{rel} = 0.05$};
\addplot [color=black,only marks,mark=*,mark options={solid},forget plot]
  table[row sep=crcr]{%
1	22.2\\
8	34\\
9	35\\
10	39.5\\
11	38.5\\
};
\addplot [color=black,solid,mark=o,mark options={solid}]
  table[row sep=crcr]{%
0	16\\
1	22.2\\
2	21\\
3	23\\
4	25\\
5	27\\
6	33.5\\
7	35\\
8	38.6666666666667\\
9	33\\
10	33.5\\
11	32\\
12	34\\
13	35\\
14	36.5\\
15	36\\
16	39\\
17	41\\
18	38\\
19	40\\
20	39\\
21	36\\
22	36\\
23	33\\
24	42\\
25	13\\
26	50\\
27	12\\
28	50\\
};
\addlegendentry{$H_\mathrm{rel} = 0.1$};
\addplot [color=black,only marks,mark=*,mark options={solid},forget plot]
  table[row sep=crcr]{%
1	22.2\\
6	33.5\\
7	35\\
8	38.6666666666667\\
10	33.5\\
14	36.5\\
};
\end{axis}
\end{tikzpicture}%
  \caption{The average number of GMRES iterations for a GMRES--Newton method
  with backward step control for $\kappa = 10^{-2}$ and varying values of
  $H_\mathrm{rel}$ for the Carrier equation
  with $\varepsilon = 10^{-3}$. Only in the iterations marked with $\bullet$,
  the backward step control bisection procedure performs extra iterations to
  determine $t_{k-1}$ and $\delta u_k$ and the plotted GMRES iterations are
  averaged over the number of bisection steps.}
  \label{fig:carrier_gmres}
\end{figure}
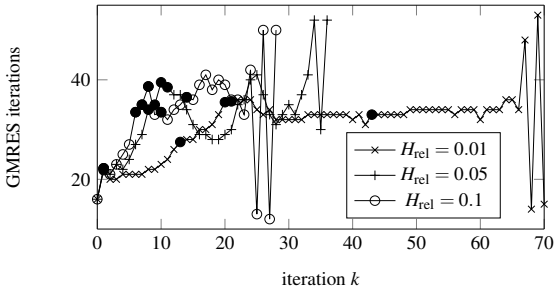

\subsection{Approximation by Finite Elements}
\label{sec:FE}

In contrast to section~\ref{sec:PGMRES}, we now explicitly discretize the increment
$\delta u_k$ and the step determination equation by Finite Elements: To this
end, let $\mathcal{C}$ be a partition of $\Omega$ into cells $C \in
\mathcal{C}$. We can then construct the finite-dimensional Finite Element
subspace
\[
U_{\mathcal{C}}^p = \left\{ u \in C^0(\Omega) \mid u \text{ is a polynomial
of degree $p$ on each } C \in \mathcal{C} \right\} \subset U.
\]
The increment is then determined by finding $\delta u_k \in U_{\mathcal{C}}^p$
such that
\begin{equation}
  \label{eqn:discretizedincrement}
  (F'(u_k) \delta u_k) \varphi = -F(u_k) \varphi \quad 
  \text{for all } \varphi \in U_{\mathcal{C}}^p.
\end{equation}
By fixing a nodal basis 
of $U_{\mathcal{C}}^p$, we obtain a linear system with a large but sparse
(typically symmetric positive definite)
$\abs{U_{\mathcal{C}}^p}$-by-$\abs{U_{\mathcal{C}}^p}$ matrix.

\subsubsection{Computation of norms in $V$}

The discretization in the previous paragraph is completely standard. We now
describe the non-standard part, which comprises the computation of 
\begin{equation}
  \label{eqn:kappak}
  \norm{F(u_k)}_V \quad \text{and} \quad 
  \kappa_k := \frac{\norm{F(u_k) + F'(u_k) \delta u_k}_V}{\norm{F(u_k)}_V}
\end{equation}
in the space $V = H^{-1}(\Omega)$. To this end, we use again the Riesz
representation \eqref{eqn:RieszRepresentation}. However, using the same Finite
Element subspace $U_{\mathcal{C}}^p$ for the discretization of
\eqref{eqn:RieszRepresentation} would yield the wrong value $\kappa_k = 0$
because the numerator vanishes. Moreover, using $U_{\mathcal{C}}^p$ would also
give wrong results for $\norm{F(u_k)}_V$ because the residual projected on
$U_{\mathcal{C}}^p$ converges locally quadratic (as a Newton method on a
finite-dimensional space) if we solve \eqref{eqn:discretizedincrement} exactly,
but does not see the discretization error. Thus, we need to solve
\eqref{eqn:RieszRepresentation} on richer Finite Element spaces. The numerical
results in section~\ref{sec:minsurf} indicate that choosing $U_{\mathcal{C}}^{p+1}$
seems to be sufficient for good estimates of the required $V$-norms. 

\subsubsection{Adaptive mesh refinement to minimize the contraction rate
$\kappa$}

Using $p$-refinement for the solution of \eqref{eqn:RieszRepresentation} instead
of refinement of $\mathcal{C}$ has two advantages: First, the increase in the
degrees of freedom $\abs{\smash[b]{\smash[t]{U_{\mathcal{C}}^{p+1}}}}$ with
respect to $\abs{U_{\mathcal{C}}^p}$ is only moderate if $p$ is moderately
large, e.g., $p = 3$. 
Second, the squared norm $\smash[t]{\norm{r_v}}_U^2$ can then be written as a sum of
contributions from each cell $C \in \mathcal{C}$, which indicate which cells
should ideally be refined if $\kappa_k$ is larger than a desired residual
contraction rate $\kappa < 1$ prescribed by the user. The cellwise contributions
$\kappa_k$ can be treated in the same fashion as existing cellwise error
indicators.

\subsubsection{Application to the minimum surface equation}
\label{sec:minsurf}

In this section, we consider the classical minimum surface problem in the
following special form: Let $\abs{.}$ denote the Euclidean norm in
$\mathbb{R}^2$ and let $\Omega = \left\{ x \in \mathbb{R}^2 \mid
\abs{x}_2 < 1 \right\}$ and $u^\partial(x) = \sin(2 \pi (x_1 + x_2))$.
We seek a function $u$ on $\Omega$ that equals $u^\partial$ on $\partial
\Omega$ and minimizes the area of its graph
\[
  \min I(u) = \int_{\Omega} \sqrt{1 + \smash[t]{\abs{\nabla u}}^2} \quad
  \text{s.t.} \quad u\big|_{\partial\Omega} = u^\partial\big|_{\partial\Omega}.
\]
With the spaces $U = H^1_0(\Omega)$ and $V = H^{-1}(\Omega)$ as before, the
minimum is described as the solution $u \in u^\partial + U$ to the variational
problem
\[
  F(u) \varphi := \int_{\Omega} \nabla \varphi \cdot \vect{g}(\nabla u)
 = 0 \quad \text{for all } \varphi \in U,
 \quad \text{where } \vect{g}(\vect{v}) := \left(1 +
 \smash[t]{\abs{\vect{v}}}^2\right)^{-\frac{1}{2}} \vect{v}.
\]
Thus, $F$ maps $u^\partial + U$ to $V$. Its G\^ateaux derivative $F':
(u^\partial + U) \times U \to V$ can then be expressed as
\[
\left( F'(u) \delta u \right) \varphi
= \int_{\Omega} \nabla \varphi \cdot \vect{g}'(\nabla u) \nabla \delta u,
\]
where the Jacobian of $\vect{g}$ is given by
\[
  \vect{g}'(\vect{v}) = \frac{1}{\sqrt{1 + \smash[t]{\abs{\vect{v}}}^2}} \left( \eye_2 -
  \frac{1}{1 + \smash[t]{\abs{\vect{v}}}^2} \vect{v} \vect{v}^T \right).
\]
We recall that $F$ is continuously G\^ateaux differentiable from
$U = H^1_0(\Omega)$ (with norm $\smash[t]{\norm{u}}_U^2 = \int_{\Omega}
\smash[t]{\abs{\nabla u}}^2$) to $V = H^{-1}(\Omega)$. To see this, 
we use the chain rule \cite[3.3.4.~Thm.]{Hamilton1982} on the continuously
G\^ateaux differentiable $L^2(\Omega, \mathbb{R}^2)$ inner product and the
Nemytskii operator defined by $\vect{g}$. The Nemytskii operator defined by
$\vect{g}$ mapping from $L^2(\Omega,
\mathbb{R}^2)$ to itself is continuously G\^ateaux differentiable by virtue of
\cite[Thm.~8, Rem.~6]{Goldberg1992}, because $\vect{g}$ and its Jacobian are
uniformly bounded $\abs{\vect{g}(\vect{v})} \le 1$ and $\abs{\vect{g}'(\vect{v})}_{2 \times 2}
\le 1$ (where $\abs{.}_{2 \times 2}$ denotes the spectral norm of 2-by-2
matrices), implied by the eigenvalues
\[
  \frac{1}{ \sqrt{1 + \smash[t]{\abs{\vect{v}}}^2}}
  \left( 1 - \frac{\smash[t]{\abs{\vect{v}}}^2}{1 + \smash[t]{\abs{\vect{v}}}^2} \right)
  \quad \text{and} \quad
  \frac{1}{\sqrt{1 + \smash[t]{\abs{\vect{v}}}^2}},
\]
of $\vect{g}'(\vect{v})$ corresponding to the eigenspace spanned by $\vect{v}$ and its
complement. We remark here that $F$ does not satisfy the stronger property of
being continuously Fr\'echet differentiable as a mapping from $H^1_0(\Omega)$ to
$H^{-1}(\Omega)$ as noted in \cite{Wachsmuth2014}.

We solve the resulting system \eqref{eqn:discretizedincrement} only
approximately with a preconditioned Conjugate Gradient (PCG) method
\cite{Hestenes1952}. The resulting inexactness also contributes to the
computations of $\kappa_k$ in \eqref{eqn:kappak}.

All computations were obtained with the software package deal.II
\cite{Bangerth2016,Bangerth2007}.

\subsubsection{Efficient computation of Riesz representations}

As described in section~\ref{sec:FE}, the Riesz representations $r_v$ need to be
computed from \eqref{eqn:RieszRepresentation} in order to compute the $V$-norms
entering $\kappa_k$. The following algorithmical and computational approaches
are important to prevent the computation times for the solution of
\eqref{eqn:RieszRepresentation} on the richer Finite Element subspace
$U_{\mathcal{C}}^{p+1}$ from dominating the overall computational effort:
\begin{enumerate}
\item We found that PCG with a symmetric Gauss--Seidel smoother as
  preconditioner delivers good results. In the computations reported below, we
  employed a Chebyshev smoother of degree four because it yields a slightly
  better performance than Gauss--Seidel when running the code on several
  processors in parallel.
  The use of multigrid does not pay off because it usually involves a rather
  expensive setup machinery on unstructured meshes (see, e.g.,
  \cite{Janssen2011}) and the right-hand sides of
  \eqref{eqn:RieszRepresentation} consist mainly of high-frequency residuals
  in $U_{\mathcal{C}}^{p+1} \setminus U_{\mathcal{C}}^p$, the low frequency
  residuals having been mostly eliminated on $U_{\mathcal{C}}^p$ already. 
\item In order to avoid a possible memory bottleneck caused by storing the
  stiffness matrix discretized on the high-dimensional space
  $U_{\mathcal{C}}^{p+1}$, we use a matrix-free realization of the Laplacian
  \cite{Kronbichler2012}.
\item Because the resulting computation times are then dominated by the
  bandwidth of the access to main memory, we perform all computations involved
  in the (matrix-free) matrix-vector products for the solution of
  \eqref{eqn:RieszRepresentation} only with single instead of double precision
  floating point arithmetic. The numerical results below indicate that this
  approach is still sufficiently accurate, while being considerably faster.
\item Because $\kappa_k$ only steers the algorithm but does not affect the
  quality of the iterates $x_k$ directly, it can be computed with rather low
  accuracy requirements in the PCG method. We use relative stopping criteria of
  $0.1$ and $0.05$ for the numerator and the denominator of $\kappa_k$ in
  \eqref{eqn:kappak}.
\end{enumerate}

\subsubsection{Details about the numerical setup}

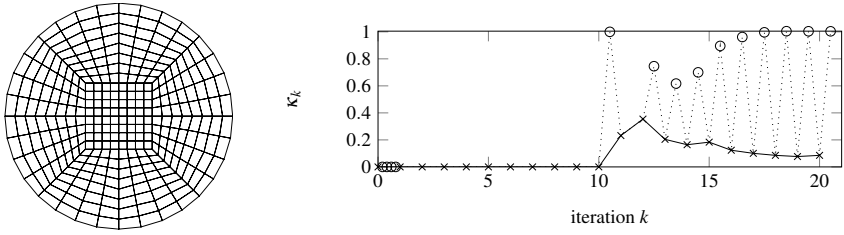
\begin{figure}[tb]
  \begin{center}
    \includegraphics[height=30mm]{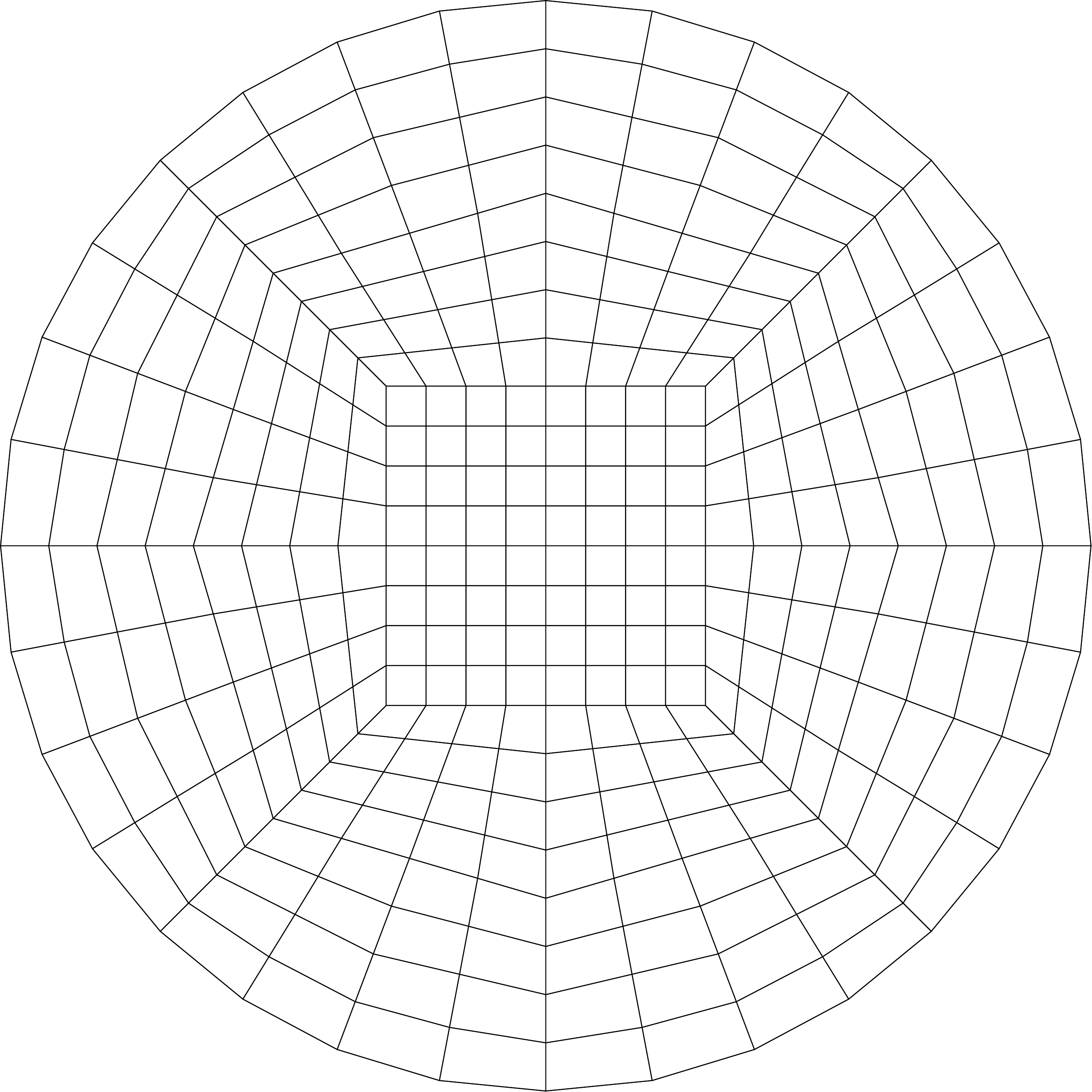}
    \qquad
%
%
\begin{tikzpicture}[scale=0.9]

\begin{axis}[%
width=0.570552\textwidth,
height=2cm,
at={(0\textwidth,0\textwidth)},
scale only axis,
xmin=0,
xmax=21,
xlabel={iteration $k$},
ymin=0,
ymax=1.005,
ylabel={$\kappa{}_k$}
]
\addplot [color=black,dotted,forget plot]
  table[row sep=crcr]{%
0	0\\
0.2	0\\
0.4	0\\
0.6	0\\
0.8	0\\
1	0\\
2	0\\
3	0\\
4	0\\
5	0\\
6	0\\
7	0\\
8	0\\
9	0\\
10	0\\
10.5	0.9999\\
11	0.232\\
12	0.354\\
12.5	0.7455\\
13	0.204\\
13.5	0.617\\
14	0.1643\\
14.5	0.7\\
15	0.1833\\
15.5	0.8957\\
16	0.1234\\
16.5	0.9616\\
17	0.1005\\
17.5	0.9963\\
18	0.0856\\
18.5	1.0047\\
19	0.0768\\
19.5	1.0041\\
20	0.0856\\
20.5	1.0038\\
};
\addplot [color=black,only marks,mark=o,mark options={solid},forget plot]
  table[row sep=crcr]{%
0.2	0\\
0.4	0\\
0.6	0\\
0.8	0\\
10.5	0.9999\\
12.5	0.7455\\
13.5	0.617\\
14.5	0.7\\
15.5	0.8957\\
16.5	0.9616\\
17.5	0.9963\\
18.5	1.0047\\
19.5	1.0041\\
20.5	1.0038\\
};
\addplot [color=black,mark=x,mark options={solid},forget plot]
  table[row sep=crcr]{%
0	0\\
1	0\\
2	0\\
3	0\\
4	0\\
5	0\\
6	0\\
7	0\\
8	0\\
9	0\\
10	0\\
11	0.232\\
12	0.354\\
13	0.204\\
14	0.1643\\
15	0.1833\\
16	0.1234\\
17	0.1005\\
18	0.0856\\
19	0.0768\\
20	0.0856\\
};
\end{axis}
\end{tikzpicture}%
  \end{center}
  \caption{Left: The initial mesh for the solution of the minimum surface
  equation has 425 cells. With $p=3$, the resulting Finite Element space has
  2929 degrees of freedom. Right: The accepted and discarded values of
  $\kappa_k$ in the numerical solution of the minimum surface equation.
  Discarded values are marked with $\circ$ at fractional iteration numbers, the
  values of $\kappa_k$ on the finally successful mesh are marked with $\times$
  at integer iteration numbers.
  The discarded values of $\kappa_k$ converge to 1, while the accepted ones
  approach approximately 0.08.}
  \label{fig:initialmesh}
\end{figure}


\begin{figure}[tb]
  \begin{center}
%
%
\definecolor{mycolor1}{rgb}{0.00000,0.00000,0.17241}%
\definecolor{mycolor2}{rgb}{1.00000,0.10345,0.72414}%
\begin{tikzpicture}[scale=0.9]

\begin{axis}[%
width=0.410625\textwidth,
height=0.397541\textwidth,
at={(0\textwidth,0\textwidth)},
scale only axis,
xmode=log,
xmin=0.1,
xmax=10000,
xminorticks=true,
xlabel={CPU time [s]},
ymode=log,
ymin=1e-06,
ymax=10,
yminorticks=true,
ylabel={residual norm $\norm{F(x_k)}_V$},
legend style={legend cell align=left,align=left,draw=white!15!black}
]
\addplot [color=black,solid,mark=x,mark options={solid}]
  table[row sep=crcr]{%
0.6569	1.4093\\
0.7173	1.2745\\
0.7749	1.1233\\
0.8321	0.83379\\
0.8903	0.53218\\
0.9459	0.24909\\
1.0065	0.11379\\
1.0572	0.060679\\
1.1095	0.044948\\
1.1785	0.042512\\
1.2565	0.042512\\
1.3295	0.076745\\
1.4216	0.018745\\
1.5347	0.018745\\
1.6518	0.0067871\\
1.7915	0.0040193\\
1.9899	0.0040193\\
2.222	0.0013314\\
2.4989	0.0012716\\
3.0801	0.0012716\\
3.8018	0.0010114\\
5.5578	0.0010114\\
8.4796	0.00068072\\
11.9542	0.00068072\\
25.766	0.0002054\\
48.7749	0.0002054\\
118.3143	5.8942e-05\\
165.1087	5.8942e-05\\
402.0966	2.0544e-05\\
566.0668	7.2458e-06\\
566.0668	7.2458e-06\\
};
\addlegendentry{$\text{Kelly }\rho\text{ = 0.5}$};

\addplot [color=black,solid,mark=+,mark options={solid}]
  table[row sep=crcr]{%
0.6509	1.4093\\
0.7049	1.2745\\
0.7606	1.1233\\
0.8186	0.8338\\
0.8759	0.53219\\
0.9305	0.24909\\
0.9822	0.1138\\
1.0329	0.060683\\
1.0835	0.044948\\
1.1475	0.042511\\
1.2236	0.042511\\
1.2951	0.076752\\
1.3648	0.018743\\
1.4269	0.013782\\
1.5163	0.012889\\
1.6318	0.012889\\
1.7524	0.0049808\\
1.8638	0.0038339\\
2.0245	0.0037813\\
2.2126	0.0037813\\
2.4259	0.0021012\\
2.6886	0.00203\\
3.2294	0.00203\\
3.9146	0.0010093\\
4.7052	0.0010038\\
6.4155	0.0010038\\
9.0932	0.00080883\\
12.022	0.00080883\\
22.2241	0.00025117\\
34.0215	0.00025143\\
40.8896	0.00025143\\
77.1263	9.2869e-05\\
101.2777	9.2869e-05\\
270.1509	2.1851e-05\\
501.3119	2.1855e-05\\
538.3801	2.1855e-05\\
792.596	8.0685e-06\\
1346.5666	8.0687e-06\\
1369.4008	6.0308e-06\\
1369.4008	6.0308e-06\\
};
\addlegendentry{$\text{Kelly }\rho\text{ = 0.1}$};

\addplot [color=black,solid,mark=o,mark options={solid}]
  table[row sep=crcr]{%
0.647	1.4093\\
0.7082	1.2745\\
0.766	1.1233\\
0.8237	0.8338\\
0.8814	0.53219\\
0.9373	0.24909\\
0.9896	0.1138\\
1.0396	0.060683\\
1.0898	0.044948\\
1.1555	0.042511\\
1.235	0.042511\\
1.3071	0.076752\\
1.3777	0.018743\\
1.44	0.013782\\
1.5059	0.012887\\
1.5854	0.012737\\
1.6907	0.012737\\
1.7998	0.0049758\\
1.9026	0.0038927\\
2.0143	0.0038462\\
2.1668	0.0038155\\
2.3504	0.0038155\\
2.563	0.0013034\\
2.761	0.0012694\\
3.0384	0.0012694\\
3.5738	0.0012694\\
4.2193	0.0010128\\
4.9588	0.0010081\\
6.5587	0.0010081\\
8.8608	0.00080769\\
11.5654	0.00081045\\
14.0906	0.00081045\\
24.5086	0.00025064\\
36.0325	0.00025091\\
42.8729	0.00025091\\
75.9669	9.2857e-05\\
124.1975	9.2878e-05\\
141.2214	9.2878e-05\\
280.0188	2.1852e-05\\
496.0503	2.1855e-05\\
532.4073	2.1855e-05\\
860.2228	8.0686e-06\\
1362.5546	8.0688e-06\\
1385.2158	6.0252e-06\\
1385.2158	6.0252e-06\\
};
\addlegendentry{$\text{Kelly }\rho\text{ = 0.01}$};

\addplot [color=mycolor1,thick,mark=diamond,mark options={solid}]
  table[row sep=crcr]{%
0.6716	1.4093\\
0.7334	1.2745\\
0.7935	1.1233\\
0.8528	0.8338\\
0.9128	0.53219\\
0.9711	0.24909\\
1.0255	0.11377\\
1.0778	0.060681\\
1.1299	0.044952\\
1.1886	0.042485\\
1.3844	0.042327\\
1.4807	0.076964\\
1.734	0.019545\\
2.1166	0.0072844\\
3.0378	0.0018734\\
5.8258	0.00066282\\
14.857	0.00022764\\
41.7841	7.8925e-05\\
108.2508	2.7796e-05\\
319.9212	9.8912e-06\\
702.2327	3.4649e-06\\
702.2327	3.4649e-06\\
};
\addlegendentry{$\kappa\text{ estimator}$};

\end{axis}

\begin{axis}[%
width=0.410625\textwidth,
height=0.397541\textwidth,
at={(0.540296\textwidth,0\textwidth)},
scale only axis,
xmode=log,
xmin=1000,
xmax=10000000,
xminorticks=true,
xlabel={degrees of freedom},
ymode=log,
ymin=1e-07,
ymax=0.1,
yminorticks=true,
ylabel={residual norm $\norm{F(x_k)}_V$},
legend style={legend cell align=left,align=left,draw=white!15!black}
]
\addplot [color=black,solid,mark=x,mark options={solid},forget plot]
  table[row sep=crcr]{%
4501	0.076745\\
6965	0.0067871\\
11963	0.0013314\\
22695	0.0012716\\
47485	0.0010114\\
122065	0.00068072\\
369401	0.0002054\\
777677	5.8942e-05\\
1527887	7.2458e-06\\
};

\addplot [color=black,solid,mark=+,mark options={solid},forget plot]
  table[row sep=crcr]{%
4501	0.013782\\
7093	0.0038339\\
11665	0.0021012\\
24039	0.0010093\\
46163	0.0010038\\
107965	0.00025117\\
272321	0.00025143\\
763645	2.1851e-05\\
1374813	8.0685e-06\\
1436135	6.0308e-06\\
};

\addplot [color=black,solid,mark=o,mark options={solid},forget plot]
  table[row sep=crcr]{%
4501	0.012887\\
6965	0.0038462\\
12151	0.0012694\\
23851	0.0010128\\
45919	0.00080769\\
108345	0.00025064\\
271861	9.2857e-05\\
763293	2.1852e-05\\
1372449	8.0686e-06\\
1436328	6.0252e-06\\
};

\addplot [color=mycolor1,thick,mark=diamond,mark options={solid},forget plot]
  table[row sep=crcr]{%
3583	0.076964\\
4877	0.019545\\
9159	0.0072844\\
19187	0.0018734\\
48633	0.00066282\\
110109	0.00022764\\
255927	7.8925e-05\\
555353	2.7796e-05\\
1004985	9.8912e-06\\
1526294	3.4649e-06\\
};

\addplot [color=black,dashed]
  table[row sep=crcr]{%
1000	0.1\\
10000000	1e-07\\
};
\addlegendentry{$\smash[t]{\text{dof}}^{\text{-3/2}}$};

\end{axis}
\end{tikzpicture}%
  \end{center}
  \caption{Convergence of the residuals for backward step control
  $\kappa$-optimizing mesh refinement (black) and mesh refinement based on the
  Kelly indicator for varying values of the required reduction $\rho$ on the
  current mesh for the minimum surface equation.}
  \label{fig:minsurfConvergence}
\end{figure}
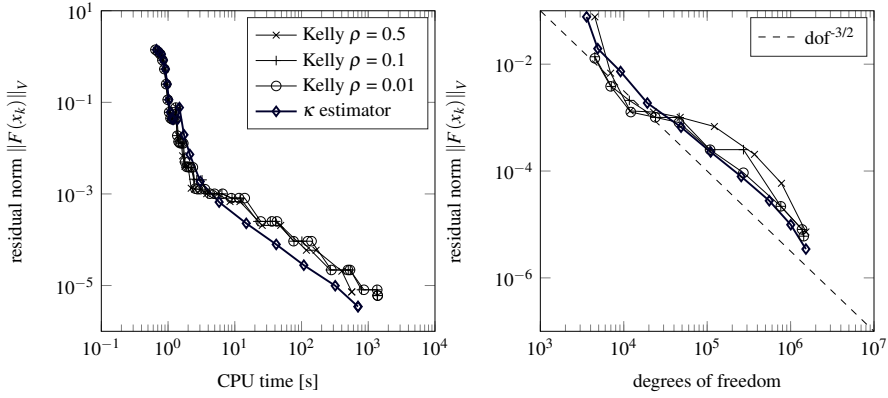

We discretize \eqref{eqn:discretizedincrement} and
\eqref{eqn:RieszRepresentation} by nodal Finite Elements of order $p = 3$ and $p
+ 1 = 4$, respectively, on quadrilaterals with tensor product polynomials using
Gauss--Lobatto nodes. For the elements on the curved boundary, we employ
polynomial tranformations of degree seven from the reference cell to the
physical cells. The relative tolerance of PCG for the solution of
\eqref{eqn:discretizedincrement} is $0.001$ (in the Euclidean norm on the
discretized vectors).

As a starting guess, we let $u_0$ be the Finite Element interpolation of
$u^\partial$ on the coarsest mesh depicted in Fig.~\ref{fig:initialmesh}. In a
first phase, we iterate until the increment norm in $U$ is below 0.01 without
computing any $V$-norms. This first phase is a finite-dimensional Newton method
($\kappa = 0$) globalized with backward step control ($H_{\mathrm{rel}} = 0.05$,
$t_0 = 1$) on the Finite Element subspace $U_{\mathcal{C}}^p$ belonging to the
initial mesh.

For the successive phase of nonlinear adaptive mesh refinement described in
section~\ref{sec:FE}, we choose $\kappa = 0.5$. If $\kappa_k > \kappa$, we mark
all cells for refinement that have a contribution of more than $2^{-p}$ times
the maximum cell contribution to $\kappa_k$, up to a given maximum number of
200,000 cells. The resulting number of cells might be slightly higher in the
final mesh due to mesh smoothing in deal.II. 

In the first eight iterations, the step size is gradually increased from $t_0 =
0.0625$ to $t_7 = 0.9738$. All other iterations are performed with full steps
$t_k = 1$. The mesh is refined for the first time in iteration 11, kept for
iteration 12, and then successively refined in each further step until the
maximum number of cells is reached in iteration 21. We can furthermore observe
from Fig.~\ref{fig:initialmesh} that from iteration 16 on, the
first trial value of $\kappa_k$ before the refinement is always one (up to three
decimal digits), which shows that no further improvement can be achieved by
performing more nonlinear iterations on the current discretization, which
is automatically detected correctly by the algorithm.

We compare in Fig.~\ref{fig:minsurfConvergence} the convergence of the residual
norm $\norm{F(u_k)}_V$ (computed afterwards with high accuracy on a refined
mesh, which is generated by one additional global refinement step of the
triangulation of the final mesh) of backward step control $\kappa$-optimizing
adaptive mesh refinement with the convergence when using mesh refinement with
the deal.II builtin Kelly error indicator \cite{Kelly1983,Gago1983}. 
In contrast to the theory of backward
step control, there is no theoretical guideline for the Kelly indicator on how
many nonlinear iterations to run before another round of refinement is
triggered. We choose to refine as soon as the increment norm $\norm{\delta
u_k}_U$ becomes less than or equal to a factor $\rho > 0$ of the error estimate
returned on the last mesh by the Kelly indicator.
Fig.~\ref{fig:minsurfConvergence} shows that $\kappa$-optimizing adaptive mesh
refinement delivers the best ratio of residual norm versus CPU time and versus
the number of degrees of freedom compared to mesh refinement based on the Kelly
indicator for varying values of $\rho = 0.5, 0.1, 0.01$. 

The computations for the solution of the minimum surface equation with a final
number of 1.5 million degrees of freedom using $\kappa$-optimizing mesh
refinement took 112\,s wall clock time on the four cores of a mid 2012 MacBook
Pro, 2.3 GHz Intel Core i7, 8 GB. Out of this grand total, the computations
necessary for estimating $\kappa_k$ took only 18\,s, even though they need to be
performed on the high-dimensional Finite Element space $U_{\mathcal{C}}^{p+1}$.

\section{Conclusions}

We presented a comprehensive convergence analysis for \eqref{eqn:BSC}, a method
that globalizes the convergence of Newton-type methods \eqref{eqn:Newton} for
the solution of \eqref{eqn:FOfxIsZero} in a Hilbert space setting. We proved
that under the reasonable assumptions A\ref{ass:validIni}--A\ref{ass:gamma} the
iterates $u_k$ either leave the region of $r$-regular points $\mathcal{R}_r$ (in
which case we need to adjust $M$ or embed $F$ in a suitable homotopy in order to
prevent attraction to
singularities) or converge to the distinctive solution $u_0^\ast$ (the initial
guess $u_0$ propagated by the generalized Newton flow \eqref{eqn:genDavidenko})
provided that $H > 0$ is chosen sufficiently small. Moreover, we provided an
$H$-dependent a priori bound on the decrease of $\norm{F(u_k)}_V$ and
characterized the asymptotic linear residual convergence rate by $\kappa$. We
provided efficient numerical methods based on the blueprint of bounding and
optimizing $\kappa$ in each iteration, either over a finite-dimensional subspace
in a
Krylov--Newton method or through an adaptive Finite Element discretization, in
order to balance the nonlinear residual norm with the residual norm of
the linear systems.  We applied these methods to the class of nonlinear elliptic
boundary value problems and presented numerical results for the Carrier equation
in a Chebfun implementation and for the minimum surface equation in deal.II. The
challenge to efficiently compute norms in $V = H^{-1}(\Omega)$ via the Riesz
preconditioner can be addressed by suitable numerical methods and techniques.

\begin{acknowledgements}
  The author is grateful to Felix Lenders 
  and to Gerd Wachsmuth for comments on an earlier draft of this manuscript and
  to the anonymous reviewers for their fruitful comments.
  This work was funded by the European Research Council through S.  Engell's and
  H.G. Bock's ERC Advanced Investigator Grant MOBOCON (291 458) and by the
  German Federal Ministry of Education and Research under grants 05M2013-GOSSIP
  and 05M2016-MOPhaPro.
\end{acknowledgements}


\end{document}